\documentclass[11pt,reqno]{amsart}
\usepackage[english]{babel}

\usepackage[a4paper,twoside]{geometry}
\usepackage{microtype}
\usepackage[shortlabels]{enumitem}
\usepackage{csquotes}
\usepackage{longtable}
\usepackage{bbm}%per la funzione indicatrice con 1

\usepackage{xspace}
\patchcmd{\subsection}{-.5em}{.5em}{}{}
\usepackage[hyperfootnotes=false]{hyperref} 
\usepackage{color}
\hypersetup{
	colorlinks = true,
	linkcolor = {blue},
	urlcolor = {blue},
	citecolor = {blue}
}

\makeatletter
\newcommand{\mysubsubsection}[1]{\subsubsection*{\bfseries #1}}
\renewcommand{\tocsection}[3]{
  \indentlabel{\@ifnotempty{#2}{\ignorespaces#1 #2\quad}}\bfseries#3}
\renewcommand{\tocsubsection}[3]{
  \indentlabel{\@ifnotempty{#2}{\ignorespaces#1 #2\quad}}#3}

\newcommand\@dotsep{4.5}
\def\@tocline#1#2#3#4#5#6#7{\relax
  \ifnum #1>\c@tocdepth
  \else
    \par \addpenalty\@secpenalty\addvspace{#2}
    \begingroup \hyphenpenalty\@M
    \@ifempty{#4}{
      \@tempdima\csname r@tocindent\number#1\endcsname\relax
    }{
      \@tempdima#4\relax
    }
    \parindent\z@ \leftskip#3\relax \advance\leftskip\@tempdima\relax
    \rightskip\@pnumwidth plus1em \parfillskip-\@pnumwidth
    #5\leavevmode\hskip-\@tempdima{#6}\nobreak
    \leaders\hbox{$\m@th\mkern \@dotsep mu\hbox{.}\mkern \@dotsep mu$}\hfill
    \nobreak
    \hbox to\@pnumwidth{\@tocpagenum{\ifnum#1=1\bfseries\fi#7}}\par
    \nobreak
    \endgroup
  \fi}
\AtBeginDocument{
\expandafter\renewcommand\csname r@tocindent0\endcsname{0pt}
}
\def\l@subsection{\@tocline{2}{0pt}{2.5pc}{5pc}{}}
\makeatother

\usepackage[backend=biber,style=alphabetic,bibencoding=utf8,language=auto,autolang=other,
url=false, doi=false, eprint=false, isbn=false, maxnames=10, giveninits=true]{biblatex}

\usepackage{amsmath}
\usepackage{amsthm}
\usepackage{amssymb}
\usepackage{mathrsfs} 
\usepackage{eucal}
\usepackage{mathtools}

\usepackage{graphicx}
\usepackage{tikz}
\usetikzlibrary{cd}

\newcounter{results}[section]

\theoremstyle{plain}
\newtheorem{theorem}[results]{Theorem}
\newtheorem{lemma}[results]{Lemma}
\newtheorem{proposition}[results]{Proposition}
\newtheorem{corollary}[results]{Corollary}

\theoremstyle{remark}
\newtheorem{remark}[results]{Remark}
\newtheorem{example}[results]{Example}

\theoremstyle{definition}
\newtheorem{definition}[results]{Definition}

\numberwithin{equation}{section}

\newcommand{\R}{\ensuremath{\mathbb R}} % Real numbers
\newcommand{\N}{\ensuremath{\mathbb N}} % Natural numbers

\newcommand{\eps}{\ensuremath{\varepsilon}} % Epsilon
\newcommand{\restr}[1]{\lower3pt\hbox{$|_{#1}$}} %restrizione funzione

\newcommand{\dom}{\ensuremath{\mathrm{D}}}% Domain
\renewcommand{\P}{\mathbb P} % reference probability measure on Omega
\newcommand{\E}{\mathbb E} % mean

\newcommand{\upds}{{\frac{\d}{\d s}}^{\kern-3pt +}} % Upper right Dini derivative with respect to s
\newcommand{\updt}{{\frac{\d}{\d t}}^{\kern-3pt +}} % Upper right Dini derivative with respect to t
\newcommand{\lodt}{{\frac{\d}{\d t}}_{\kern-1pt +}} % Lower right Dini derivative with respect to t

\newcommand{\la}{\langle} % <
\newcommand{\ra}{\rangle} % >

\newcommand{\de}{\ensuremath{\,\mathrm d}} % The differential for integrals (\de x)
\renewcommand{\d}{\mathrm d}  % The differential for integrals (\de x)

\newcommand{\bb}{\boldsymbol b} % Bold b
\newcommand{\ff}{\boldsymbol f}

\newcommand{\Bb}{\boldsymbol B} % Bold capital B

\newcommand{\Aa}{{\boldsymbol A}}
\newcommand{\hAa}{{\boldsymbol {\hat A}}}

\renewcommand{\gg}{\boldsymbol g}
\newcommand{\ggamma}{\boldsymbol\gamma} % Bold gamma
\newcommand{\ii}{\boldsymbol i} % Bold i
\newcommand{\jj}{\boldsymbol j} % Bold i
\newcommand{\jJ}{\boldsymbol J} % Bold i
\newcommand{\mmu}{\boldsymbol\mu} % Bold mu
\renewcommand{\ss}{\boldsymbol s} % Bold u
\renewcommand{\H}{\mathcal H} % Serif capital h
\newcommand{\X}{\mathsf X} % Serif capital X
\newcommand{\Y}{\mathsf X} % Serif capital y
\newcommand{\Z}{\mathsf Y} % Serif capital y

\newcommand{\indi}{\mathrm{I}}  % indicator function

\newcommand{\cN}{{\mathfrak N}} % directed subset of integers
\newcommand{\cB}{{\mathcal B}} % Borel sigma-algebra

\newcommand{\fu}[1]{f_{#1}} % Lebesgue measure
\newcommand{\prob}{\ensuremath{\mathcal{P}}} % Space of probability measures
\DeclareMathOperator{\supp}{supp} % Support of a measure

\newcommand{\TX}{\mathsf {T\kern-1.5pt X}} % Tangent space to X
\newcommand{\TY}{\mathsf {T\kern-1.5pt Y}} % Tangent space to Y

\newcommand{\rmC}{\mathrm C} % Space of continuous functions
\newcommand{\scalprod}[2]{\ensuremath{\langle #1, #2\rangle}} % Scalar product
\DeclareMathOperator{\cl}{cl}
\newcommand{\clconv}[1]{\overline{\operatorname{co}}\left(#1\right)} % Closed convexification of a MPVF
\newcommand{\conv}[1]{\operatorname{co}(#1)} % Convex hull

\newcommand{\rmS}{\mathrm S} % Essentially injetive maps

\newcommand{\symg}[1]{{\mathrm{Sym}(#1)}} % set for permutations

\newcommand{\resolvent}[1]{{\jJ_{#1}}}

\newcommand{\mm}{\mathfrak m}
\newcommand{\Sgp}{\boldsymbol S}
\makeatletter
\newcommand{\precneq}{\mathrel{\text{\prec@eq}}}
\newcommand{\prec@eq}{%
  \oalign{%
    \hidewidth$\m@th\prec$\hidewidth\cr
    \noalign{\nointerlineskip\kern1ex}%
    $\m@th\smash{\Neq}$\cr
    \noalign{\nointerlineskip\kern-.5ex}%
  }%
}
\newcommand{\Neq}{\raisebox{0.65ex}{\rotatebox{90}{\scalebox{1}[-1]{$\nshortmid$}}}}
\makeatother
\newcommand{\cY}{\mathcal X}
\newcommand{\cZ}{\mathcal Y}
\newcommand{\cH}{\mathcal X}
\newcommand{\Sp}[1]{\mathcal{S}(#1)}
\newcommand{\euler}[1]{{\boldsymbol f}_{#1}}
\newcommand{\map}{{\boldsymbol L}}
\newcommand{\G}{\mathsf G}
\newcommand{\W}{\mathsf W}
\newcommand{\dommmo}{\iota\big(D(\Bb)\big)}

\title[Properties of invariant Lipschitz maps and dissipative operators]{Extension 
of monotone operators  
and Lipschitz maps invariant for a group of isometries}

\author{Giulia Cavagnari}
\address{Giulia Cavagnari: Politecnico di Milano, Dipartimento di Matematica, Piazza Leonardo Da Vinci 32, 20133 Milano (Italy)}
\email{giulia.cavagnari@polimi.it}

\author{Giuseppe Savar\'e}
\address{Giuseppe Savar\'e: Bocconi University,
  Department of Decision Sciences and BIDSA, Via Roentgen 1, 20136 Milano (Italy)}
\email{giuseppe.savare@unibocconi.it}

\author{Giacomo Enrico Sodini}
\address{Giacomo Enrico Sodini: Institut für Mathematik - Fakultät für Mathematik - Universität Wien, Oskar-Morgenstern-Platz 1, 1090 Wien (Austria)}
\email{giacomo.sodini@univie.ac.at}


\begin{document}

\begin{abstract}
  We study 
  monotone operators 
  in reflexive Banach spaces
  that are invariant with respect to a group 
  of 
  suitable isometric isomorphisms and we show that they always admit a maximal extension which preserves the same invariance. 
  A similar result applies to Lipschitz maps
  in Hilbert spaces, thus providing an 
  invariant version of Kirzsbraun-Valentine
  extension Theorem.
  
  We then provide a relevant application to the case of monotone operators in $L^p$-spaces of random variables which are invariant with respect to measure-preserving isomorphisms,
  proving that they always admit maximal dissipative extensions 
  which are still invariant by measure-preserving isomorphisms. 
  We also show that such  operators
  are law invariant, a much stronger property which is also inherited by their  resolvents, the Moreau-Yosida approximations, and the associated semigroup of contractions. 

  These results combine explicit representation
  formulae for the maximal extension 
  of a monotone operator based on selfdual lagrangians and a refined study of measure-preserving maps in standard  Borel spaces endowed with a nonatomic measure, 
  with applications to the approximation of arbitrary couplings between measures by 
  sequences of maps.
\end{abstract}

\maketitle
\tableofcontents
\thispagestyle{empty}

\section{Introduction}

The theory of maximal monotone operators
$\Aa:\cY\rightrightarrows \cY^*$
in Hilbert and reflexive Banach spaces
provides a very powerful framework to solve 
nonlinear equations (see e.g.~the review \cite{Bor10}).
We recall that an operator $\Aa\subset\cY\times\cY^*$ (which we identify with its graph) is said to be \emph{monotone} if
\[\scalprod{v-w}{x-y}\ge0\quad\text{for any }(x,v),\,(y,w)\in\Aa,\]
while $\Aa$ is said to be \emph{maximal monotone} if every proper extension of $\Aa$ fails to be monotone.

In the Hilbertian case, 
the theory can also be applied to 
differential inclusions of the form
\begin{equation}
\label{eq:diff}
    \frac\d{\d t} x(t)\in -\Aa x(t),\quad
    x(0)=x_0
\end{equation}
driven by a maximal monotone operator $\Aa$
and to prove the generation of a semigroup of contractions (see e.g.~\cite{BrezisFR,Barbu10}).

The notion of 
maximality 
of the (multivalued) operator $\Aa$ plays a crucial role, since 
by Minty-Browder theorem it is 
equivalent to 
the solvability of 
the resolvent equation
\begin{equation}
    \label{eq:resolvent}
    J(x)+\tau\Aa x\ni y,
\end{equation}
where $J$ is the duality map 
from $\cY$ to $\cY^*$
 \cite[Thm.~2.2]{Barbu10}.
In the Hilbertian framework,
the solution to \eqref{eq:resolvent}
corresponds to 
the solvability of the 
Implicit Euler Scheme
associated to \eqref{eq:diff}
and provides a general condition
for the existence of a solution to \eqref{eq:diff}.
In this respect, an essential tool is the well known 
fact that 
every monotone operator $\Aa$ admits a maximal extension \cite{DbF64,BrezisFR},
whose domain is contained
in the closed convex hull of
the domain of $\Aa.$

\medskip\noindent
Motivated by the study
of operators in Bochner-$L^p$ spaces
$\cY=L^p(\Omega,\cB,\P;\Y)$ (here $(\Omega, \cB)$ is a standard Borel space endowed with a nonatomic probability measure $\P$ and $\X$ is a reflexive and separable Banach space)
which are invariant by measure-preserving transformations of $\Omega$,
in this paper we address the general problem of 
finding maximal 
extensions of monotone
operators which are invariant by
a group $\G$ of 
suitable transformations of 
$\cY\times \cY^*$. 
More precisely, let us 
consider a group 
$\G$ of linear isomorphisms
acting on $\cY\times \cY^*$
whose elements $\mathsf U=(U,U')$
preserve 
the duality pairing and
the norms in $\cY\times \cY^*$,
i.e.~for every $(U,U')\in \G$
and every $z=(x,v)\in \cY\times \cY^*$
we have
\begin{equation}
    \label{eq:group}
    \langle U'v,Ux\rangle=
    \langle v,x\rangle,\quad
    |Ux|=|x|,\quad
    |U'v|_*=|v|_*.
\end{equation}
Given a monotone operator $\Aa\subset 
\cY\times \cY^*$
which is $\G$-invariant, i.e.
\begin{equation}
    \label{eq:Ginv-intro}
    (x,v)\in \Aa,\quad
    (U,U')\in \G\quad\Rightarrow\quad
    (Ux,U'v)\in \Aa,
\end{equation}
we will prove (see Theorem \ref{thm:graziebauinv}) that there exists 
a maximal extension $\hAa$ 
of $\Aa$
preserving the $\G$-invariance; we will also find $\hAa$
so that its 
proper domain $\dom(\hAa)$ 
does not exceed
the closed convex hull of 
$\dom(\Aa).$

Since it is not clear
how to adapt to this
context the classical proof
based on the Debrunner-Flor 
and Zorn Lemma
(see e.g.~\cite[Thm.~2.1 and Cor.~2.1, Chap.II]{BrezisFR}),
we will use the 
powerful explicit construction
of \cite{BauWang2009}. This is
based on 
kernel averages of convex functionals
and on the characterization
of monotone and 
maximal monotone operators
via suitable convex Lagrangians 
on $\cY\times \cY^*$,
a deep theory started with
the seminal paper 
\cite{Fitzpatrick88} (where the so called Fitzpatrick's function is introduced for the first time) and further developed in a more recent 
series of
relevant contributions
(see e.g.~\cite{MlT01,Burachik-Svaiter03,Penot04,Penot04cras,Ghoussoub08,Visintin17} and the references therein).

The advantage of this direct approach
is that 
it provides 
an explicit formula 
for the extension of $\Aa$ which
behaves quite well with respect to
the action of the group $\G$.
 As an intermediate step, which can be relevant also in other applications independently
of $\G$-invariance,  
we will also show (Theorem \ref{thm:main1}) how to modify
the construction of 
\cite{BauWang2009}
in order to confine
the domain of the extension
$\hAa$
to the closed convex hull
of $\dom(\Aa)$ (see also [Theorem 2.13]\cite{BauWang2010} for a partial result in this direction).

As a byproduct,
we can adapt the 
same strategy of 
\cite{BauWang2010}
to prove a version of the
Kirszbraun-Valentine 
extension theorem (see \cite{Kirszbraun34,Valentine43,Valentine45}) for 
$\G$-invariant Lipschitz maps
in Hilbert spaces (Theorem \ref{thm:inv-lip-ext}):
it states that every $L$-Lipschitz map $f:D\to\H$ 
defined in a subset $D$ of an Hilbert space $\H$, 
whose graph is invariant
with respect to the action of a group
$\G$ of isometries of $\H$,
can be extended to 
a $L$-Lipschitz function $\hat f:\H\to \H$ which is $\G$-invariant as well.
The basic idea here
still goes back to Minty:
the graphs of nonexpansive
maps in Hilbert spaces
are in one-to-one correspondence
with graphs of 
monotone maps via
the Cayley transformation
$T:\H\times \H\to \H\times \H$ defined as
\begin{equation}
    \label{eq:minty-rotations-intro}
    T(y,w):=\frac 1{\sqrt 2}(y-w,y+w).
\end{equation}

It is worth noticing that 
such a correspondence allowed the authors of \cite{RS05} to use for the first time the Fitzpatrick's function to prove the Kirszbraun-Valentine theorem (see \cite{Kirszbraun34,Valentine43,Valentine45}), which states that every $1$-Lipschitz continuous map can be extended to the whole $\H$ (see also \cite{B07} where this approach is improved in order to obtain an extension with an optimal range). 
The same correspondence, together with the explicit construction of a maximal extension of a monotone operator $\Aa$ in \cite{BauWang2009}, is used in \cite{BauWang2010} to provide the first constructive proof of the Kirszbraun-Valentine theorem.

We will also provide in the appendix an alternative
proof based on another more
recent explicit formula 
 for such kind of extension 
 given by
 \cite{ALM21} (see also \cite{ALM18}).

\medskip

\medskip

These results, besides being interesting by themselves, find interesting  applications in the case 
when
$\cY$ is the $L^p$-space of random variables
\begin{equation}\label{eq:HL2}
    \cY=L^p(\Omega,\cB,\P;\Y),\quad
    \cY^*=L^{p^*}(\Omega, \cB,\P;\Y^*),\quad
    p,p^*\in (1,+\infty), \ \frac 1p+\frac 1{p^*}=1,
\end{equation}
over a space of parametrizations $(\Omega,\cB,\P)$, where $(\Omega,\cB)$ is a standard Borel space, $\P$ is a nonatomic probability measure, and $\Y$ is a separable and reflexive Banach space, while $\G$ is 
a group of isomorphisms generated by \emph{measure-preserving maps}, i.e. $\cB-\cB$ measurable maps $g:\Omega\to\Omega$ which are essentially injective and such that $g_\sharp\P=\P$, where $g_\sharp\P$ denotes the push-forward of $\P$ by $g$. 
Every measure-preserving isomorphism $g$ 
induces 
an element $\mathsf U_g$ of $\G$ 
whose action on $(X,X')\in \cY\times \cY^*$ is simply given by $\mathsf U_g(X,X')=
(X\circ g,X'\circ g)$. 

The interest for invariance by measure-preserving isomorphisms 
in $\cY\times \cY^*$ is justified by
its link with the stronger property 
of \emph{law invariance}: a set $\Aa\subset 
\cY\times \cY^*$ is law invariant if 
whenever $(X,X')\in \Aa$ 
then $\Aa$ also contains all the pairs $(Y,Y')\in \cY\times \cY^*$ 
with the same law of $(X,X')$, i.e.~$(Y,Y')_\sharp\P=(X,X')_\sharp\P.$
It is clear that law invariant subsets of 
$\cY\times \cY^*$ are also invariant by measure-preserving isomorphisms; using the results of \cite{Brenier-Gangbo03},
we will show that the converse implication holds for closed sets: therefore for closed sets these two properties are in fact equivalent. 
Since the graph of a maximal monotone operator
is closed, we obtain that
a monotone operator in $\cY\times \cY^*$
whose graph is invariant by the action of
measure-preserving isomorphisms
admits a maximal monotone extension
which is law invariant (Theorem \ref{thm:maximal-monotonicity}).

This framework is exploited in Section \ref{sec:Borel} (where we study 
the approximation of transport maps and plans 
by various classes of measure-preserving isomorphisms) and 
Section \ref{subsec:invariant-maps}.

\medskip
The Hilbertian setting when $p=p^*=2$ and 
$\Y$ is an Hilbert space 
(so that $\cY=L^2(\Omega,\cB,\P;\Y)$ 
is a Hilbert space as well that can be 
identified with its dual $\cY^*$) 
provides an important 
case, which we will further exploit in \cite{CSS2grande}.
It turns out that maximal dissipative operators $\Bb$ on $L^2(\Omega,\cB,\P;\Y)$, invariant by measure-preserving isomorphisms, 
are the Hilbertian counterparts of maximal totally dissipative operators on the Wasserstein space $\prob_2(\Y)$ of laws, where $\prob_2(\Y)$ denotes the space of Borel probability measures with finite second moment 
 endowed with the so-called Kantorovich-Rubinstein-Wasserstein distance $W_2$. The results obtained in Section \ref{subsec:invariant-maps} in the framework \eqref{eq:HL2} are used in \cite{CSS2grande} to develop a well-posedness theory for dissipative evolution equations in the metric space $(\prob_2(\Y),W_2)$, together with a Lagrangian characterization for the solution of the corresponding Cauchy problem.
\medskip

Besides the direct application of the general invariance extension result provided in Section \ref{sec:appC}, in Section \ref{subsec:invariant-maps} we also analyze further properties of Lipschitz functions and maximal dissipative operators on $\cY= L^2(\Omega,\cB,\P;\Y)$, which are invariant by measure-preserving isomorphisms. 
In particular we prove 
that 
the effect of a Lipschitz invariant map $\map:\cY\to\cY$ 
on an element $X\in \cY$ 
can always be represented as 
\[\text{$\map X(w)=l(X(w),X_\sharp\P)$ 
\quad
for a.e.~$\omega\in
    \Omega$}\]
where
$l:\Sp{\Y}\to\Y$ is a (uniqueley determined) continuous map defined in
\[\Sp{\Y}:=\Big\{(x,\mu)\in \Y\times \prob_2(\Y)\,:\,x\in \supp(\mu)\Big\}\]
whose sections $l(\cdot,\mu)$ are Lipschitz as well, for every $\mu\in \prob_2(\Y).$

An important application of these results concerns the resolvent operator, the Moreau-Yosida approximation, and the semigroup 
associated with a maximal dissipative invariant operator $\Bb$ in $L^2(\Omega,\cB,\P; \X)$, for which we
obtain new relevant representation formulae
(Theorem \ref{thm:invTOlawinv}). 

\medskip
The above structural characterizations rely on
various approximation properties 
for couplings between probability measures in terms of maps and measure-preserving transformations.
We collect them in Section \ref{sec:Borel}, 
with the aim to present
many important results available in the literature (cf. \cite{Brenier-Gangbo03,CD18,gangbotudo,pra}) 
in a unified framework and (in some cases) a slightly more general setting
adapted to Section \ref{subsec:invariant-maps}.

\mysubsubsection{Acknowledgments.}
 G.S.~and G.E.S.~gratefully acknowledge the support of the Institute for Advanced Study of the Technical University of Munich, funded by the German Excellence Initiative.
 G.C.~and G.S. have been supported by the MIUR-PRIN 2017
project \emph{Gradient flows, Optimal Transport and Metric Measure Structures}.  G.C. also acknowledges the partial support of INDAM-GNAMPA project 2022 \emph{Evoluzione e controllo ottimo in spazi di Wasserstein} (CUP\_E55F22000270001) and
the funds FSR Politecnico di Milano Prog.TDG3ATEN02. G.S.~also thanks IMATI-CNR, Pavia.

\section{Extension of monotone operators 
and Lipschitz maps
invariant by a group of isometries}\label{sec:appC}
\newcommand{\sfc}{\mathsf c}
\newcommand{\cW}{\mathcal Z}
Let $\cY$ be a reflexive Banach space with norm $|\cdot|$ and let $\cY^*$ be its dual endowed with the dual norm $|\cdot|_*$.

We denote by 
$\sfc:\cY^*\times \cY\to \R$, 
 the duality pairing 
 $\la \cdot, \cdot \ra$ between  $\cY^*$ and $\cY$ and by
$\cW$ the product space $\cY\times \cY^*$ with dual
$\cW^*:=\cY^*\times \cY$.
 
A (multivalued) operator 
$\Aa:\cY\rightrightarrows \cY^*$  (which we identify with
its graph, a subset of 
$ \cY \times \cY^*$)
is monotone if it satisfies
\begin{equation}
    \label{eq:monotonicity}
    \la v-w, x-y \ra \ge 0 \quad \text{ for every } (x,v), (y,w) \in \Aa.
\end{equation}
The proper domain $\dom(\Aa)\subset \cY$ 
of $\Aa$ is just the projection on the first component
of (the graph of) $\Aa.$
A monotone operator $\Aa$ is maximal if any monotone operator in $\cY\times \cY^*$ containing $\Aa$ coincides with $\Aa.$

In order to address the extension problem
of monotone operators $\Aa \subset \cY \times \cY^*$ invariant by the action of 
a group of isometric isomorphisms, 
it is crucial to have
some explicit formula 
providing a maximal extension of $\Aa$.
In this respect, 
the characterization of
monotone and maximal monotone operators
by means of suitable ``contact sets'' of convex functionals in 
$\cY\times \cY^*,$ started with
the seminal paper 
\cite{Fitzpatrick88} and further developed in a more recent 
series of
relevant contributions
(see e.g.~\cite{MlT01,Burachik-Svaiter03,Penot04,Penot04cras,Ghoussoub08,Visintin17} and the references therein),
and the kernel averaging 
operation developed by \cite{BauWang2009}
provide extremly powerful tools,
that we are going to quickly recall in the next section.
We will also show how to slightly improve
this construction in order 
to obtain an explicit formula
providing a maximal extension of $\Aa$ whose
domain is contained in the closed convex hull of $\dom(\Aa)$. In this connection, we mention that the existence of a maximal extension of $\Aa$ with the desired above-mentioned optimality for the domain can be deduced by \cite{B07}, thanks to the correspondence revealed by Minty between monotone operators and \emph{firmly nonexpansive mappings}. Indeed, \cite{B07} uses the Fitzpatrick function to prove the Kirszbraun-Valentine extension theorem for firmly non-expansive mappings with optimal range localization. However, part of the proof still relies on Zorn's Lemma and it is not entirely constructive.

\subsection{Maximal extensions of monotone operators by self-dual Lagrangians}
\label{subsec:Fitz}
\renewcommand{\fu}[1]{\sfc_{#1}}
\newcommand{\Fitz}[1]{\mathsf f_{#1}}
\newcommand{\Penot}[1]{\mathsf p_{#1}}
\newcommand{\frs}{\mathfrak s}
Following the presentation of \cite{BauWang2009} and \cite{Penot04}, 
given a set $\Aa \subset \cY \times \cY^*$ and its
indicator function
$\indi_\Aa$,
we consider the proper function 
$\fu\Aa: \cY^* \times \cY \to (-\infty, +\infty]$ defined as
\[ \fu\Aa(v,x):=
\sfc(v,x)+\indi_\Aa(x,v)=
\begin{cases} \la v, x \ra 
    \quad &\text{ if } (x,v) \in \Aa, \\ + \infty \quad &\text{ else}.\end{cases}\]
Notice that $\fu \Aa$ has an affine minorant if $\Aa$ is monotone, in the sense that
\[ \fu \Aa (v,x) \ge 
\la v_0,x \ra + \la v,x_0\ra - \la v_0, x_0 \ra
\quad \text{ for every } (v,x) \in \cY^* \times \cY,\]
where $(x_0,v_0)\in \Aa$ is an arbitrary given point.

Recalling that 
the \emph{convex conjugate} 
$g^*:\cW^* \to (-\infty,+\infty]$
of a proper function
$g: \cW \to (-\infty, +\infty]$ 
with an affine minorant
is defined as 
\begin{align*}
    g^*(v,x)&:= \sup_{(x_0,v_0) \in \cW} \left \{ \la v_0,x \ra + \la v,x_0\ra -g(x_0,v_0)\right \}, \quad (v,x) \in \cY^* \times \cY,
\end{align*}
with an analogous definition
in the case of 
a function $h:\cW^*  \to (-\infty,+\infty]$, 
we can introduce
the Fitzpatrick function
$\Fitz\Aa:\cW\to (-\infty,+\infty]$ and 
the convex l.s.c.~relaxation 
$\Penot\Aa:\cW^*\to (-\infty,+\infty]$ of
$\fu\Aa$:
\begin{equation}
    \label{eq:main-funct}
    \Fitz\Aa:=\fu\Aa^*,
    \quad
    \Penot{\Aa}:=\Fitz\Aa^*=
    \fu{\Aa}^{**}.
\end{equation}

 It will be often useful to switch the order 
 of the components of 
 elements in $\cY\times \cY^*$: 
 we will denote by $\frs:\cY\times \cY^*
 \to \cY^*\times \cY$ the switch map 
 \begin{equation}
     \label{eq:switch}
     \frs(x,v):=(v,x).
 \end{equation}
 If  $g$ is any function defined 
 in $\cY \times \cY^*$ (resp.~in $\cY^*\times \cY$)
 we set $g^\top:=g\circ\frs$
 (resp.~$g^\top:=g\circ\frs^{-1}$).
 In particular $\sfc^\top$ is the duality pairing
 between $\cY$ and $\cY^*$.
 
 We collect in the following statement some useful properties.

 \begin{theorem}[Representation of monotone operators]
     \label{thm:omnibus}
     \ 
     \begin{enumerate}[\rm (1)]
         \item If $f:\cW\to (-\infty,+\infty]$
         is a convex l.s.c.~function
         satisfying 
         $f\ge \sfc^\top$,
         then 
         \begin{equation}
             \label{eq:contact}
            \text{the contact set }
            \Aa_f:=\Big\{(x,v)\in \cW:
            f(x,v)=\langle v,x\rangle \Big\}
            \text{ is monotone,  }
            f^*\ge \Fitz{\Aa_f}^\top,
         \end{equation}
         and 
         \begin{equation}
         \label{eq:switched-subd}
             \Aa_f\subset 
             \boldsymbol{T}_f:=\Big\{(x,v)\in \cW:
             (v,x)\in \partial f(x,v)\Big\},
         \end{equation}
          where $\partial f$ denotes the subdifferential of $f$.
         An analogous statement
         holds for $g:\cW^*\to(-\infty,+\infty]$
         satisfying $g\ge \sfc$,
         by setting $\Aa_g:=\Aa_{g^\top}.$
         \item If $\Aa$ is monotone,
         then 
         \begin{gather}
         \label{eq:1}
             \Fitz\Aa^\top\le 
             \Penot\Aa\le 
             \sfc_\Aa;\\
             \label{eq:22}
             \Penot\Aa\ge \sfc;\quad
             \Penot\Aa=\sfc\quad
             \text{on }\Aa,          
         \end{gather}
        i.e.~$\Aa\subset \Aa_{\Penot\Aa}.$
        \item 
        If $\Aa$ is a monotone operator and $f:\cW\to (-\infty,+\infty]$
        is a convex l.s.c.~function
        satisfying 
        $\sfc^\top\le f\le \Penot\Aa^\top$,
        then the contact set $\Aa_f$
        is a monotone extension of $\Aa.$
        \item 
        If $\Aa \subset \cY \times \cY^*$ is maximal monotone 
        then 
        $\Fitz\Aa\ge \sfc^\top$. 
        Conversely, if $\Aa$ is monotone and
        $\Fitz\Aa\ge \sfc^\top$,
        then 
        \begin{equation}
            \hAa:=
            \Big\{z\in \cW:\Fitz\Aa(z)=
            \sfc^\top (z)\Big\}
            =
            \Big\{z\in \cW:
            \Penot\Aa(z)=
            \sfc (z)\Big\}
        \end{equation}
        provides a maximal monotone extension of $\Aa$.
        \item
        If $f:\cW\to (-\infty,+\infty]$
         is a convex l.s.c.~function
         satisfying 
         $f\ge \sfc^\top$,
         then $\Aa_f$
         is maximal monotone if and only if $f^*\ge \sfc$
         and in this case $\Aa_f=
         \Aa_{(f^*)^\top}.$
         \item 
         If 
         $f:\cW\to (-\infty,+\infty]$
         is a convex l.s.c.~function
         satisfying the self-duality property
         \begin{equation}
             \label{eq:self-duality}
             f^*=f^\top,
         \end{equation}
         then $f\ge \sfc^\top$ and 
         the 
         contact set $\Aa_f$
         is maximal monotone.
     \end{enumerate}
 \end{theorem}
 \begin{proof}
 We give a few references and sketches for the proof.
 
\medskip
(1) 
The inclusion in \eqref{eq:switched-subd}
follows by \cite[Theorem 2.4]{Fitzpatrick88} 
and shows in particular  
that $\Aa_f$ is monotone since $\mathbf{T}_f$ is monotone by \cite[Proposition 2.2]{Fitzpatrick88} . Clearly $f^\top \le \sfc_{\Aa_f}$ so that, passing to the conjugates, the reverse inequality follows.    

\medskip
(2) The result in \eqref{eq:1} is \cite[Proposition 4(c)]{Penot04} while \eqref{eq:22} is \cite[Proposition 4(f)]{Penot04}.

\medskip
(3) This follows by (1) and (2).

\medskip
(4) The first implication follows by \cite[Theorem 5]{Penot04}. 
 The second implication follows by \cite[Theorem 6]{Penot04} choosing $g:=f_\Aa$ and \cite[Theorem 3.1]{Burachik-Svaiter03}.

\medskip
(5) The first part of the sentence is \cite[Theorem 6]{Penot04} while the equality $\Aa_f=\Aa_{(f^*)^\top}$ can be found e.g.~in\cite[Theorem 3.1]{Burachik-Svaiter03}.

\medskip
(6) This is contained in the statement and in the proof of \cite[Fact 5.6]{BauWang2009}.

 \end{proof}

 The previous result suggests 
 a strategy (cf. Theorem \ref{thm:bau} below) to construct 
 a maximal extension of a given monotone operator $\Aa$ starting
 from a convex and l.s.c.~function $f:\cW\to (-\infty,+\infty]$
 satisfying 
 \begin{equation}
     \label{eq:representation}
     \Fitz\Aa\le f\le \Penot\Aa^\top\quad\text{in }\cW.
 \end{equation}
 Using the kernel average introduced
 in \cite{BauWang2009}
 one obtains a self-dual
 Lagrangian $R_f$ 
 which satisfies $\Fitz\Aa\le 
 R_f\le \Penot\Aa^\top$,
 so that the contact set of $R_f$ 
 is a 
 maximal monotone extension of $\Aa.$
 We introduce a function
 \begin{equation}
 \label{eq:kernel}
    \text{$\psi:\cW\to[0,+\infty)$ satisfying
 the simmetry and self-duality condition} 
 \quad
 \psi=\psi^{\vee}=(\psi^*)^\top,
 \end{equation}
 where $\psi^\vee(z):=\psi(-z)$ for every $z \in \mathcal Z.$
The assumption in \eqref{eq:kernel} in particular implies 
that 
\begin{equation}
    \label{eq:psi-regularity}
    \psi\text{ is continuous, convex, and }
    \psi(0)=\psi^*(0)=0,
\end{equation}
 since
\begin{displaymath}
    0\le \psi(0)=\psi^\top(0)=
    \psi^*(0)=-\inf \psi\le 0.
\end{displaymath}
A typical example is given by
\begin{equation}
\label{eq:ex-kernel}
   \psi(x,v):=\frac 1p|x|^p+\frac{1}{p^*}|v|_*^{p^*} 
   \quad\text{where $p,p^*\in (1,+\infty)$ are given conjugate exponents.}
\end{equation}
 The following result is an immediate consequence of \cite[Fact 5.6, Theorem 5.7, Remark 5.8]{BauWang2009}.
\begin{theorem}[Kernel averages and maximal monotone extensions \cite{BauWang2009}]
\label{thm:bau}
Let $\Aa \subset \cY \times \cY^*$ be a monotone operator and let $\Fitz\Aa:=\fu\Aa^*, \Penot\Aa:=\fu\Aa^{**}$ be defined as above
and $\psi:\cW\to[0,+\infty)$ a self-dual function
as in \eqref{eq:kernel}.
Let $f: \cW \to (-\infty,+\infty]$ be a  lower semicontinuous and convex function satisfying \eqref{eq:representation}.
\begin{enumerate}[\rm (1)]
    \item 
    The function
$R_f: \cW \to (-\infty,+\infty]$ 
defined as
\begin{equation}\label{eq:rf}
    R_f(x,v):= \min_{(x,v)=\frac{1}{2}(x_1+x_2,v_1+v_2)} \left \{ \frac{1}{2}f(x_1,v_1) + \frac{1}{2}f^{*}(v_2,x_2)+\frac{1}{4}
    \psi(x_1-x_2,v_1-v_2) 
    \right \},
\end{equation}
 is 
self-dual and satisfies the bound
\eqref{eq:representation}, i.e.~$\Fitz\Aa\le R_f\le\Penot\Aa^\top.$
\item The operator $\tilde{\Aa}$ defined as the contact set of $R_f$
\begin{equation}\label{eq:btilde}
\tilde{\Aa} := \left \{ (x,v) \in 
\cW : R_f(x,v)=\la v,x\ra\right \},
\end{equation}
 is a 
maximal monotone extension of $\Aa$.
\end{enumerate}
\end{theorem}
We want to show that, for a suitable choice of $f$ as in the previous theorem, we can produce a maximal monotone extension of $\Aa$ with domain included in $ D:=\clconv{\dom(\Aa)}$. 
We claim that such $f$ can be defined as
\begin{equation}\label{eq:f0}
f(x,v):= \Fitz\Aa(x,v)+ \indi_C(x,v), \quad (x,v)
\in \cY \times \cY^*,\quad
C:=D\times \cY^*,
\end{equation}
where $\indi_C$ is the indicator function of $C=D\times \cY^*$ i.e. $\indi_C(x,v)=0$ 
if $x \in D$ and $\indi_C(x,v)=+\infty$ if $x \notin D$.

\newcommand{\episum}{\mathbin\Box}
\begin{theorem}[A maximal monotone extension with minimal domain]
    \label{thm:main1}
    Let $\Aa \subset \cY \times \cY^*$ be a 
    monotone operator and let 
    $f:\cW\to (-\infty, + \infty]$
    be as in \eqref{eq:f0}. The following hold:
    \begin{enumerate}[\rm (1)]
        \item 
    $f$ is a 
    l.s.c.~and convex function such that $\Fitz\Aa \le f \le \Penot\Aa^\top $;
    \item $\dom(f)\subset C=D\times \cY^*$ and $
\dom(f^*) \subset \frs(C)=\cY^*\times D$;
\item 
defining $R_f$ as in \eqref{eq:rf}
and its contact set $\tilde\Aa$
as in \eqref{eq:btilde}, 
$\dom(R_f)\subset C=D\times \cY^*$ and
$\tilde \Aa$ provides a maximal extension of
$\Aa$ with domain $\dom(\tilde\Aa)\subset D.$
    \end{enumerate}
\end{theorem}
\begin{proof}
(1)
It is clear that $f$ is convex and lower semicontinuous and
$f \ge \Fitz\Aa$. 
On the other hand, $f^\top\le \sfc_\Aa$ by \eqref{eq:1} and 
since $\indi_C=0$ on $\Aa$.
It follows that $f^\top\le \sfc_\Aa^{**}=\Penot\Aa.$

\medskip\noindent
(2)
It is clear from the definition of $f$ that $f(x,v)=+\infty$ if $x \notin D$. Let us compute the conjugate $f^*$ of $f$: by Fenchel-Rockafellar duality theorem, (see e.g.~\cite[Theorem 16.4]{Rockafellar}) 
we have that
\[f^*=(\Fitz\Aa+\indi_C)^* = \cl(\Fitz\Aa^{*}\episum \indi^*_C)=
\cl(\Penot\Aa\episum \indi^*_C )\]
where $\cl$ denotes the 
lower semicontinuous envelope of 
a given function
(i.e.~$\cl(h)$ is 
the largest lower semincontinuous function staying below $h$)
and $\episum$ denotes the 
inf-convolution (or epigraphical sum) of two functions: if 
$k,j:\cW^* \to (-\infty, + \infty]$ then
$k\episum j$ is defined as
\[ (k\episum j)(z):=\inf_{z_1,z_2\in \cW^*,
z_1+z_2=z} k(z_1)+j(z_2), \quad 
z\in \cW^*.\]

Since $\indi_C(x,v)=\indi_D(x)$, 
we can easily compute 
its dual for every $z=(v,x)\in \cY^* \times \cY$
\begin{align*}
    \indi_C^*(v,x)&= \sup_{(x_0,v_0) \in \cY^* \times \cY} \la v,x_0 \ra + \la v_0,x\ra -\indi_D(x_0)  
    = \sup_{x_0\in D, v_0 \in \cY^*} 
    \la v_0,x \ra + \la v,x_0\ra\\
    &= \indi_0(x)+\sigma_D(v),
\end{align*}
where $\indi_0$ is the indicator function of the singleton $\{0\}$ and 
$\sigma_D$ is the support function of $D$ defined by 
\[ \sigma_D(v):= \sup_{x_0 \in D} \la v,x_0 \ra, \quad v \in \cY^*.\]
We thus have
\begin{align*}
    (\Penot \Aa \episum 
    \indi^*_C)(v,x) &= 
    \inf_{x_1+x_2=x,v_1+v_2=v} 
    \Penot\Aa(v_1,x_1)+\indi_0(x_2)
    +\sigma_D(v_2)
    \\
    &= \inf_{v_1+v_2=v} 
    \Penot\Aa(v_1,x)+\sigma_D(v_2).
\end{align*}
Since $\Penot \Aa=\sfc_\Aa^{**}$
and $\sfc_{\Aa}(v,x)=+\infty$ if $x\not\in D$,
we deduce that $\Penot\Aa(v,x)=+\infty$ if $x\not\in D$,
$\dom(\Penot\Aa \episum \indi^*_C)\subset 
\frs(C)= \cY^*\times D$, 
and therefore that $\dom(f^*)\subset 
\cl( \cY^*\times D)=\cY^* \times D$.

\medskip\noindent
(3)
    By the first claim 
    and Theorem \ref{thm:bau} with $f$ as in \eqref{eq:f0}, 
    we obtain that $\tilde{\Aa}$ is a maximal monotone extension of $\Aa$. 
    We only need to check that $\dom(\tilde{\Aa}) \subset D=\clconv{\dom(\Aa)}$. 
    Since it is clear from the definition
    of contact set 
    that $\dom(\tilde\Aa)\subset 
    \pi^{\cY}(\dom(R_f))$, it is sufficient to check 
    that $\dom(R_f)\subset C$.
    
    Let $(x,v) \in \dom(R_f)$, then by \eqref{eq:rf} we can find $x_1,x_2 \in \cY$ and $v_1,v_2 \in \cY^*$ such that $(x,v)=\frac{1}{2}(x_1+x_2,v_1+v_2)$ and 
    $(x_1,v_1)\in \dom(f)$,
    $(v_2,x_2)\in \dom (f^*)$:
    in particular $x_1,x_2$ belong to the convex set $D$ by (2) and
    therefore $x\in D$ as well.
\end{proof}

\subsection{Extension of monotone operators
invariant w.r.t.~the action
of a group of linear isomorphisms of \texorpdfstring{$\cY\times \cY^*$}{XxX}}
\label{subsec:invext}
We will now focus on operators which are invariant
with respect to a group $\G$ of 
bounded linear isomorphisms 
of $\cW$ 
of the form $\mathsf U=(U,U'):\cW \to \cW$.
For every $z=(x,v)\in \cW$ we thus have
\begin{equation}
    \mathsf U(z)=(Ux,U'v)
\end{equation}
and we assume that 
all the maps $\mathsf U\in \G$ 
satisfy the following properties
\begin{equation}
    \label{eq:inv-duality}
    \sfc^\top( \mathsf U z)=\sfc^\top(z),\quad
    \psi(\mathsf Uz)=\psi(z)
    \quad\text{for every }z\in \cW
\end{equation}
for a fixed self-dual function 
$\psi:\cW\to [0,+\infty)$ as in \eqref{eq:kernel}.
Notice that the first identity in 
\eqref{eq:inv-duality} implies
\begin{equation}
    \langle U'v,Ux\rangle=
    \langle v,x\rangle\quad\text{for every }
    (x,v)\in \cW
\end{equation}
so that $U^*\circ U'$ (resp.~$(U')^*\circ U$)
is the identity in $\cY^*$ (resp.~in $\cY$),
i.e.~$U'=(U^*)^{-1}=(U^{-1})^*$ is the 
transpose inverse of $U$.

\newcommand{\mtop}{{\kern1pt\text{-}\kern-1pt{\scriptscriptstyle \top}}}

Given $\mathsf U=(U,U') \in \G$, we define as usual
$\mathsf U^\top:=\frs \circ \mathsf U\circ \frs^{-1}=(U',U):\cW^*\to\cW^*$ 
observing that $\mathsf U^\top$ coincides with 
the inverse transpose of $\mathsf U$ with respect to 
 the duality paring between
 $z^*=(v^*,x^*)\in \cW^*$ and
 $z=(x,v)\in \cW$ given by 
 \begin{align*}
     \langle z^*,z\rangle&=
     \langle v^*,x\rangle+
     \langle v,x^*\rangle=
     \sfc(v^*,x)+\sfc^\top(x^*,v),
 \end{align*}
 since 
 \begin{align*}
     \langle \mathsf U^\top z^*,\mathsf Uz\rangle
     &=
     \langle (U'v^*,Ux^*),
     (Ux,U'v)\rangle=
     \langle U'v^*,Ux\rangle
     +\langle U'v,Ux^*\rangle
     =
     \langle v^*,x\rangle
     +\langle  v,x^*\rangle
     \\&=
     \langle z^*,z\rangle.
 \end{align*}
 In particular, we have the formula
 \begin{equation}
     \label{eq:crucial-inv-trans}
     \langle \mathsf U^\top z^*,z
     \rangle=\langle z^*,\mathsf U^{-1}z\rangle
     \quad\text{for every }
     z\in \cW,\ z^*\in \cW^*,\ \mathsf U\in \G.
 \end{equation}
\begin{definition}\label{def:Ginv}[$\G$-invariance]
We say that a set $\Aa \subset \cY \times \cY^*$ is $\G$-invariant if $\mathsf U \Aa\subset \Aa$
for every $\mathsf U\in \G$ 
(i.e.~$(Ux,U'v)\in \Aa$ for every $(x,v) \in \Aa$ and $(U,U')\in \G$). 
A function $g$ defined in $ \cY \times \cY^*$
(resp.~in $\cY^*\times \cY$) 
is said to be $\G$-invariant if $g
\circ \mathsf U=g$ (resp.~$g\circ 
\mathsf U^\top=g$) 
for every $\mathsf U\in \G$. A function $h: \cY \to \cY^*$ is $\G$-invariant if its graph is $\G$-invariant as a subset of $\cY \times \cY^*$.
\end{definition}

The following simple result 
clarifies 
the relation between $\G$-invariance of
monotone operators and 
$\G$-invariance of the corresponding Lagrangian functions.
\begin{proposition}\label{prop:fBginv}
\ 
\begin{enumerate}[\rm (1)]
    \item 
    If
    $f:\cW\to (-\infty,+\infty]$ is $\G$-invariant,
 then 
 $f^*$ is $\G$-invariant.
 \item If 
 $\Aa \subset \cY \times \cY^*$ is a $\G$-invariant monotone operator, 
 then the functions 
 $\sfc_\Aa, \Fitz\Aa, \Penot\Aa$ 
 defined in Section \ref{subsec:Fitz}
 are $\G$-invariant.
 \item 
 If $f:\cW\to (-\infty,+\infty]$ is a $\G$-invariant, 
 l.s.c.~and convex function
 satisfying $f\ge \sfc^\top$, 
 then the contact set $\Aa_f$ defined as 
 in \eqref{eq:contact} is $\G$-invariant.
 \item 
   If
    $f:\cW\to (-\infty,+\infty]$ is $\G$-invariant, then the kernel average $R_f$ defined
 as in \eqref{eq:rf} is $\G$-invariant as well.
\end{enumerate}
\end{proposition}
\begin{proof}
(1) 
We simply have for every $\mathsf U\in \G$
\begin{align*}
    f^*(\mathsf U^\top z^*)
    &=
    \sup_{z\in \cW}\left\{
    \langle 
    \mathsf U^\top z^*,z\rangle
    -f(z)\right\}=
    \sup_{z\in \cW}\left\{
    \langle z^*,\mathsf U^{-1} z\rangle
    -f(z)\right\}
    =
    \sup_{z\in \cW}\left\{
    \langle z^*,\mathsf U^{-1} z\rangle
    -f(\mathsf U^{-1} z)\right\}
    \\&=
    \sup_{\tilde z\in \cW}\left\{
    \langle z^*,\tilde z\rangle
    -f(\tilde z)\right\}
    =f^*(z^*)
\end{align*}
where we applied 
\eqref{eq:crucial-inv-trans}, 
the fact that $\mathsf U^{-1}\in \G$,
 the $\G$-invariance of $f$ and the
fact that $\{\tilde z=\mathsf U^{-1}z:z\in 
\cW\}=\cW.$

\noindent\medskip
(2) 
One immediately sees that $\sfc_\Aa$
is $\G$ invariant 
thanks to the $\G$-invariance of $\Aa$ and
\eqref{eq:inv-duality}.
The invariance of $\Fitz\Aa$
and of $\Penot\Aa$ then follows by the previous claim.

\noindent\medskip
(3) 
If $z=(x,v)\in \Aa_f$,
we know that $f(z)=\sfc^\top(z)$.
Since $f$ is $\G$-invariant, \eqref{eq:inv-duality} yields for every $\mathsf U\in \G$
\begin{equation*}
    f(\mathsf Uz)=
    f(z)=\sfc^\top(z)=
    \sfc^\top(\mathsf Uz)
\end{equation*}
so that $\mathsf Uz\in \Aa_f.$

\noindent\medskip
(4) 
We first observe that the function
\begin{equation*}
    P(z_1,z_2):=
    \frac 12 f(z_1)+
    \frac 12 f^*(\frs(z_2))+
    \frac 14 \psi(z_1-z_2)
\end{equation*}
satisfies
\begin{equation}
    \label{eq:invarianceP}
    P(\mathsf Uz_1,\mathsf Uz_2)=
    P(z_1,z_2),
\end{equation}
thanks to the invariance of $f$,
the invariance of $f^*$ from claim (1),
and the invariance property of 
$\psi$ stated in \eqref{eq:inv-duality}.

Since $\mathsf U$ is a linear isomorphism
we also have
\begin{equation*}
    z=\frac 12 z_1+\frac 12 z_2
    \quad\Leftrightarrow\quad
    \mathsf Uz=\frac 12 
    \mathsf Uz_1+\frac 12 \mathsf Uz_2.
\end{equation*}
Combining the above identities,
we immediately get $R_f(\mathsf Uz)=
R_f(z).$
\end{proof}

We can now obtain our main result.
\begin{theorem}[$\G$-invariant maximal monotone extensions]
\label{thm:graziebauinv} 
Let $\Aa \subset \cY \times \cY^*$ be a $\G$-invariant monotone operator
with $D:=\clconv{\dom(\Aa)}.$
Then the function $f$ given by \eqref{eq:f0}
is $\G$-invariant and, defining 
$R_f$ as in \eqref{eq:rf} and its  contact set
$\tilde{\Aa}$ as in \eqref{eq:btilde}, then  $\tilde \Aa$ is a $\G$-invariant maximal monotone extension of $\Aa$ 
with domain included in $\clconv{\dom(\Aa)}$.   
\end{theorem}
\begin{proof}
Let us first observe that 
$C':=\dom(\Aa)\times \cY^*$
is $\G$-invariant, thanks to the invariance of $\Aa.$ 
Since every element of $\G$
is linear, also 
$\conv{C'}$ is $\G$-invariant. 
Eventually, since every element 
of $\G$ is continuous, 
$C:=\cl(\conv{C'})$ is $\G$-invariant as well.

By claim (2) of Proposition
\ref{prop:fBginv}, 
we deduce that the function $f$ 
given by \eqref{eq:f0} is $\G$-invariant.
The invariance of $\tilde\Aa$ 
then follows by applying 
Claims (4) and (3) of 
Proposition
\ref{prop:fBginv}, recalling that $R_f\ge \sfc^\top$ by Theorem \ref{thm:omnibus}(6).
We conclude by Theorem \ref{thm:main1}.
\end{proof}

\subsection{Extension of invariant dissipative operators in Hilbert spaces}
\label{subsec:extension-diss}
 We now quickly apply 
 the results of the previous section
 \ref{subsec:invext} 
 to the particular case of 
 dissipative operators in Hilbert spaces.
 We adopt the dissipative viewpoint
 in view of applications to 
 differential equations, 
 but clearly all our statements
 apply to monotone operators as well. 
 The main reference is \cite{BrezisFR}.

 We consider a Hilbert space $\H$ with norm $|\cdot|$, scalar product $\scalprod{\cdot}{\cdot}$, and dual $\H^*$ which we
 identify with $\H$.
A multivalued  operator $\Bb\subset
\H\times \H$ is dissipative 
if the operator
\begin{equation}
    \label{eq:diss-mon}
    \Aa=-\Bb:=\Big\{(x,-v):(x,v)\in \Bb\Big\} 
\end{equation}
is monotone. 
More generally,
$\Bb$
is said to be \emph{$\lambda$-dissipative} ($\lambda \in \R$) if
\begin{equation}
  \label{eq:140}
  \la v-w,x-y\ra\le\lambda |x-y|^2\quad \text{for every }(x,v),\ (y,w)\in \Bb.
\end{equation}
\begin{remark}[$\lambda$-transformation]\label{rem:transff} 
Denoting by $\ii(\cdot)$ the identity function on $\H$,
it is easy to check that $\Bb$ is $\lambda$-dissipative if and only if $\Bb^{\lambda}:= \Bb-\lambda \ii$ is dissipative, or, equivalently,
$-\Bb^\lambda=\lambda\ii-\Bb$
is monotone.
Notice that $\dom(\Bb)=\dom(\Bb^\lambda)=\dom(-\Bb^\lambda)$.
\end{remark}

We say that a $\lambda$-dissipative operator $\Bb$ is \emph{maximal} if 
$\Bb$ is maximal 
w.r.t.~inclusion in the class of $\lambda$-dissipative operators or, equivalently, if
$-\Bb^\lambda$ is a maximal monotone operator.

If $\Bb\subset\H\times\H$ is a $\lambda$-dissipative operator, a \emph{maximal $\lambda$-dissipative extension} of $\Bb$ is any set $\boldsymbol{C} \subset \H \times \H$ such that $\Bb \subset \boldsymbol{C}$ and $\boldsymbol{C}$ is maximal $\lambda$-dissipative.

As an immediate application of Theorem \ref{thm:graziebauinv}, we 
obtain an important result for dissipative operators which are invariant with respect to 
the action of a group $\G$ of isometries.
 \begin{theorem}[Extension of invariant $\lambda$-dissipative operators]
 \label{thm:inv-max-ext}
 Let $\G_\H$ be a group of linear isometries
 of $\H$ and let $\G:=\{(U,U):U\in \G_\H\}$
 be the induced group of linear isometries
 in $\H\times \H$.
    Let $\Bb\subset \H\times \H$ be a $\lambda$-dissipative operator which is $\G$-invariant
    (as a subset of $\H\times \H$).
    Then there exists 
    a maximal $\lambda$-dissipative extension 
    $\hat \Bb$ of
    $\Bb$ 
    with $\dom(\hat\Bb)\subset 
    \clconv{\dom(\Bb)}$ 
    which is $\G$-invariant as well.
\end{theorem}
\begin{proof} 
We can choose the self-dual kernel
\begin{equation*}
    \psi(x,v):=\frac 12 |x|^2+\frac 12 |v|^2
\end{equation*}
and we immediately get that 
$\G$ satisfies 
\eqref{eq:inv-duality}.
The statement is then an immediate
application of Theorem 
\ref{thm:graziebauinv}
to the $\G$-invariant monotone operator
$\Aa:=-\Bb^\lambda.$
\end{proof}
\begin{remark}
    If $\Bb$ is a $\G$-invariant  $\lambda$-dissipative operator
    which is maximal (with respect to inclusion) in the collection of 
    $\G$-invariant $\lambda$-dissipative operators, then $\Bb$ is maximal $\lambda$-dissipative.
\end{remark}

We now derive a few more properties concerning the resolvent, the Yosida regularization
and the minimal selection 
of a maximal $\lambda$-dissipative and $\G$-invariant operator $\Bb\subset \H\times \H$.
For references on the corresponding definitions and properties we refer to \cite{BrezisFR}, where the theory is developed in detail for the case $\lambda=0$. If $\lambda\ne0$ analogous statements can be obtained and we refer the interested reader to \cite[Appendix A]{CSS2grande}.

In the following we use the notation $\lambda^+:= \lambda \vee 0$ and we set $1/\lambda^+=+\infty$ if $\lambda^+=0$. 
We denote by $\Bb x\equiv \Bb(x):=\{v\in \H:(x,v)\in
\Bb\}$ the sections of $\Bb$, with $x\in\H$.

 Recall that 
 for every $0<\tau<1/\lambda^+$
 the resolvent
 $\resolvent\tau:=(\ii-\tau\Bb)^{-1}$ 
 of $\Bb$ is a $(1-\lambda \tau)^{-1}$-Lipschitz map defined on the whole $\H$ .
The minimal selection $\Bb^\circ:\dom(\Bb)\to\H$ of $\Bb$ is also characterized by
\begin{equation*}\Bb^\circ x=\displaystyle\lim_{\tau \downarrow 0}\frac{\resolvent\tau x-x}{\tau}.
\end{equation*}
The Yosida approximation of $\Bb$ is defined by $\Bb_\tau:=\frac{\resolvent\tau-\ii}{\tau}$.
For every $0<\tau<1/\lambda^+$, $\Bb_\tau$ is maximal $\frac{\lambda}{1-\lambda \tau}$-dissipative and $\frac{2-\lambda \tau}{\tau(1-\lambda \tau)}$-Lipschitz continuous. 

Moreover (cf. \cite[Prop. 2.6]{BrezisFR} or \cite[Appendix A]{CSS2grande}), the following hold
\begin{align}
\label{eq:btaucircD}
    &\text{if $x \in \dom(\Bb)$, $\,\,(1-\lambda \tau)|\Bb_\tau x| \uparrow |\Bb^\circ x|$, as $\tau \downarrow 0$,}
    \\
    \label{eq:btaucircN}
    &\text{if $x \notin \dom(\Bb)$, $\,\,|\Bb_\tau x| \to + \infty$, as $\tau \downarrow 0$.}
\end{align}

\medskip

Since $\Bb$ is a maximal $\lambda$-dissipative operator, there exists a semigroup of $e^{\lambda t}$-Lipschitz transformations $(\Sgp_t)_{t \ge
  0}$ with $\Sgp_t: \overline{\dom(\Bb)} \to
\overline{\dom(\Bb)}$ s.t.~for every $x_0\in \dom(\Bb)$
the curve $t \mapsto \Sgp_t x_0$ is included in $\dom(\Bb)$ and it
is the unique locally Lipschitz continuous solution of the differential inclusion
\begin{equation}\label{eq:lagrsystem}
\begin{cases}
  \dot{x}_t \in \Bb x_t \quad \text{ a.e. } t>0, \\
  x \restr{t=0} = x_0.
\end{cases}
\end{equation}

We also have
\begin{equation}\label{eq:paramev}
  \lim_{h\downarrow0}\frac{\Sgp_{t+h} x_0-\Sgp_t x_0}h =
  \Bb^\circ(\Sgp_tx_0),\quad \text{for every $x_0\in \dom(\Bb)$ and every $t\ge 0$}
\end{equation}
and
\begin{equation}\label{eq:expf}
    \Sgp_tx = \lim_{n \to + \infty} (\jJ_{t/n})^n x, \quad x \in \overline{\dom(\Bb)}, \, t \ge 0.
\end{equation}

\begin{proposition}[Invariance of resolvents, Yosida regularizations,
semigroups and minimal selections]
\label{prop:resdiscr}Let $\Bb \subset \H \times \H$ be
  a maximal $\lambda$-dissipative operator which is $\G$-invariant.
  Then for every $0<\tau<1/\lambda^+,\,t\ge 0$
  the operators $\resolvent \tau$, $\Bb_\tau$, $\Sgp_t$, $\Bb^\circ$ are
  $\G$-invariant. 
\end{proposition}
\begin{proof} The identities $\resolvent\tau (Ux)=U(\resolvent\tau x)$ and
  $\Bb^\circ (Ux) = U(\Bb^\circ x)$ come from the
  $\G$-invariance of $\Bb$ and the uniqueness
  property of the resolvent operator.
  The exponential formula (cf.\eqref{eq:expf})
  \begin{equation*}
    \Sgp_t x=\lim_{n\to\infty} (\resolvent{t/n})^n(x)
  \end{equation*}
  yields the $\G$-invariance of $\Sgp_t$.
\end{proof}

\subsection{Extension of invariant Lipschitz maps in Hilbert spaces }
\label{subsec:invariant-Lip}

We conclude this general discussion
concerning $\G$-invariant sets and maps
addressing the problem of 
the global extension 
of a $\G$-invariant Lipschitz map 
defined in a subset of a 
separable Hilbert space $\H.$

As in Theorem \ref{thm:inv-max-ext},
we consider a group $\G_\H$ of isometric isomorphisms of $\H$ inducing
 the group $\G:=\{(U,U):U\in \G_\H\}$
 in $\H\times \H.$

In addition to Definition \ref{def:Ginv}, we also give the following definition.

\begin{definition}\label{def:Ginvmap}[$\G_\H$-invariance for $\H$-valued maps]
We say that a function $f:D\to \H$, where $D\subset \H$, is \emph{$\G_\H$-invariant} if 
its graph is $\G$-invariant:
\begin{equation}
\label{eq:G-invariance-map}
    \text{for every $x\in D$ 
we have $Ux\in D$ and $f(Ux)=Uf(x)$, for every $U\in\G_\H$.}
\end{equation}
\end{definition}

The Kirszbraun-Valentine Theorem 
\cite{Kirszbraun34,Valentine45}
states that 
every Lipschitz function $f:D\to \H$
defined in a subset $D$ of $\H$
with Lipschitz constant $L\ge0$
admits a Lipschitz extension $F:\H\to \H$
whose Lipschitz constant coincides with $L$. 
We want to prove that it is possibile to find such an extension preserving $\G$-invariance.

Our starting point 
is the following well known fact, going back to Minty (see also 
\cite{Alberti-Ambrosio99}).
We introduce the 
isometric Cayley transforms
$T, T^{-1}:\H\times \H\to \H\times \H$
\begin{equation}
    \label{eq:minty-rotations}
    T(y,w):=\frac 1{\sqrt 2}(y-w,y+w),\quad
    T^{-1}(x,v):=
    \frac 1{\sqrt 2}(x+v,-x+v)
\end{equation}
\begin{lemma}[Lipschitz and monotone graphs]
    \label{le:Minty-rotation}
    Let $\boldsymbol F,\Aa$
    be two subsets of $\H\times \H$
    such that $\Aa=T(\boldsymbol F)$.
    The following two properties are equivalent:
    \begin{enumerate}[\rm (1)]
        \item 
        $\boldsymbol F$
        is the graph of 
        a nonexpansive map 
    $f$ 
    defined on the set $D$ given by
    \[D:=\pi^1(\boldsymbol F)
    =\big\{y\in\H:(y,w)\in\boldsymbol F
    \text{ for some }w\in \H\big\};\]
    \item 
    $\Aa$ is monotone.
\end{enumerate}
Moreover, assuming that $\Aa$ is monotone, the following hold
\begin{enumerate}[(i)]
    \item 
$\Aa$ is maximal monotone
if and only if $D=\H.$
\item 
$\Aa$ is $\G$-invariant if and
only if $\boldsymbol F$
is $\G$-invariant
(or, equivalently, $f$ is $\G_\H$-invariant).
\item 
If $\Aa'$ is a monotone extension
of $\Aa$ 
and $f':D'\to \H$ 
is the nonexpansive map 
associated with $ \boldsymbol F':=T^{-1}(\Aa')$,
then $D':=\pi^1(\boldsymbol F')\supset D$
and $f'$ is an extension of $f.$
\end{enumerate}
\end{lemma}
\begin{proof}
    Let us take
    a pair of elements
    $(x_i,v_i)=T(y_i,w_i)\in \Aa$, $i=1,2$;
    we have
    $y_i=\frac 1{\sqrt 2}
    (x_i+v_i)$, 
    $w_i=\frac 1{\sqrt 2}(-x_i+v_i)$
    so that 
    \begin{align} \label{eq:diss1}
        2|y_1-y_2|^2=
        |x_1+v_1-(x_2+v_2)|^2 
        &=
        |x_1-x_2|^2 
        +
        |v_1-v_2|^2 
        +2\langle x_1-x_2,v_1-v_2\rangle,
        \\ \label{eq:diss2}
        2|w_1-w_2|^2 
        =|-x_1+v_1-(-x_2+v_2)|^2 
        &=
        |x_1-x_2|^2 
        +
        |v_1-v_2|^2 
        -2\langle x_1-x_2,v_1-v_2\rangle,
    \end{align}
    and then
    \begin{equation} \label{eq:diss3}
      |w_1-w_2|^2
      \le 
      |y_1-y_2|^2
      \quad
      \Leftrightarrow
      \quad
      \langle x_1-x_2,v_1-v_2\rangle
      \ge 0.
    \end{equation}
    This proves the first statement.
    
    Concerning (i), 
    it is sufficient
    to recall that 
    the domain $D$ of $f$
    coincides with the image 
    of $h(\Aa)$
    where $h(x,v):=\frac 1{\sqrt 2}(x+v)$
    and we know that $\Aa$
    is maximal monotone if and only if
    such an image coincides with $\H$ (cf.~\cite[Proposition 2.2 (ii)]{BrezisFR}).

    The second property (ii) 
    is an immediate consequence
    of the invertibility of $T$ 
    and of the linearity of 
    the transformations of $\G$
    so that 
    for every $\mathsf U=(U,U)\in \G$
    we have
    $T\circ \mathsf U=
    \mathsf U\circ T$.

    Let us eventually consider claim (iii): 
    we clearly have $\boldsymbol F'=
    T^{-1}(\Aa')\supset T^{-1}(\Aa) = \boldsymbol F$
    and therefore $D'=\pi^1(\boldsymbol F')\supset \pi^1(\boldsymbol F)=D$.
    On the other hand,
    since both $\boldsymbol F'$
    and $\boldsymbol F$
    are the graph of a nonexpansive map,
    $\boldsymbol F'\cap
    (D\times \H)=
    \boldsymbol F\cap (D\times \H)$
    and therefore 
    the restriction of $f'$
    to $D$ coincides with $f$.
\end{proof}

We can now 
state our result concerning the extension of $\G$-invariant
Lipschitz maps.
\begin{theorem}[Extension of $\G$-invariant Lipschitz maps]
    \label{thm:inv-lip-ext}
    Let us suppose that $f:D\to\H$ is
    $L$-Lipschitz and $\G_\H$-invariant
    according to \eqref{eq:G-invariance-map}.
    Then 
    there exists a $L$-Lipschitz map 
    $\hat f:\H\to\H$ 
    extending $f$ which is $\G_\H$-invariant 
    as well.
\end{theorem}
\begin{proof} 
    Up to a rescaling,
    it is not restrictive to assume 
    that $L=1$ so that $f$ is nonexpansive.
    
    Let $\boldsymbol F\subset \H\times \H$
    be the graph of $f$
    and let
    $\Aa:=T(\boldsymbol F)$.
    By Lemma \ref{le:Minty-rotation}
    we know that 
    $\Aa$ is a monotone 
    $\G$-invariant operator
    
    We can now apply Theorem
    \ref{thm:inv-max-ext}
    to find a maximal monotone
    extension $\hAa$ of 
    $\Aa$ which is still $\G$-invariant.

    Setting $\boldsymbol {\hat F}:=
    T^{-1}(\hAa)$
    we can eventually apply 
    Lemma \ref{le:Minty-rotation}
    to obtain that
    $\boldsymbol {\hat F}$
    is the graph of a nonexpansive
    map $\hat f:\hat D\to \H.$
    Moreover, Claim (i)
    shows that $\hat D=\H$
    so that $\hat f$ is globally defined,
    Claim (ii) shows that $\hat f$
    is $\G_\H$-invariant,
    and Claim (iii) 
    ensures that $\hat f$ is
    an extension of $f$.
    \end{proof}

\section{Borel partitions and almost optimal couplings}\label{sec:Borel}
In this section, we collect  some useful results concerning Borel isomorphisms and partitions of standard Borel spaces. These results, besides being interesting by themselves, will also turn out to be useful in Section \ref{subsec:invariant-maps}, where we will deal with a particular group of isometric isomorphisms on Banach spaces of $L^p$-type.

We start by fixing the fundamental definitions and notations involved in the statements of the main theorems of this Section, which are presented in Section \ref{sec:laseconda}. These results concern the approximation of arbitrary couplings between probability measures by couplings which are concentrated on maps, through the action of measure-preserving transformations.
In particular, Corollary \ref{cor:marc} concerns the approximation of bistochastic measures by the graph of measure-preserving maps and it is written here in the general context of a standard Borel space $(\Omega, \cB)$ endowed with a nonatomic probability measure $\P$. 
This result relies on the analogous 
property stated for 
the $d$-dimensional Lebesgue measure in \cite[Theorem 1.1]{Brenier-Gangbo03}. In the same spirit of Corollary \ref{cor:marc}, Corollary \ref{cor:from-gangbo} provides an approximation result for the law of a pair of measurable random variables defined on $(\Omega, \cB,\P)$ with values in a pair of separable Banach spaces. A consequence of this result is the key lemma \cite[Lemma 6.4]{carda} (cf. also \cite[Lemma 5.23 p. 379]{CD18}), which states that if $X$ and $Y$ are random variables with the same law, then $X$ can be approximated by $Y$ through the action of a sequence of measure-preserving transformations.
Finally, we reported a fundamental result in Optimal Transportation Theory, concerning the equivalence between the Monge and the Kantorovich formulations  (see \cite[Theorem B]{pra}). This is the content of Proposition \ref{prop:pratelli} where the result is written using the language of random variables (cf. also \cite[Lemma 3.13]{gangbotudo}), so as to be easily recalled in Section \ref{subsec:invariant-maps}.

 In order to introduce all the technical tools used to state and prove all these properties,
in Section \ref{sec:laprima} we list some well-known facts about standard Borel spaces.

\medskip

\begin{definition}[Standard Borel spaces and nonatomic measures]
\label{def:sbs}
A \emph{standard Borel space} $(\Omega, \cB)$ is a measurable space that is isomorphic (as a measure space) to a Polish space. Equivalently, there exists a Polish topology $\tau$ on $\Omega$ such that the Borel sigma algebra generated by $\tau$ coincides with $\cB$. We say that a positive finite measure $\mm$ on $(\Omega, \cB)$ is nonatomic (also called atomless or diffuse) if $\mm(\{\omega\}) = 0$ for every $\omega \in \Omega$ (notice that $\{\omega\} \in \cB$ since it is compact in any Polish topology on $\Omega$).
\end{definition}
We notice that, being $(\Omega,\cB)$ standard Borel, 
$\mm$ is nonatomic if and only if 
for every $B\in \cB$ with $\mm(B)>0$ there exists $B'\in \cB,\ B'\subset B$,
such that $0<\mm(B')<\mm(B).$

\begin{definition}[Partitions]
If $(\Omega, \cB)$ is a standard Borel space and $N \in \N$, a family of subsets $\mathfrak P_N=\{\Omega_{N,k}\}_{k \in I_N} \subset \cB$, where $I_N:=\{0, \dots, N-1\}$, is called a \emph{$N$-partition of $(\Omega, \cB)$} if
\[\bigcup_{k\in I_N}\Omega_{N,k}=\Omega, \quad \Omega_{N,k}\cap \Omega_{N,h}=\emptyset \text{ if }h,k\in I_N,\, h\neq k.\]
\end{definition}

If $(\Omega, \cB)$ is a standard Borel space endowed with a nonatomic, positive finite measure $\mm$, we denote by $\rmS(\Omega, \cB, \mm)$ the class of
$\cB$-$\cB$-measurable maps $g:\Omega\to\Omega$ which are
essentially injective and measure-preserving, meaning that there exists a full $\mm$-measure set $\Omega_0 \in \cB$ such that $g$ is injective on $\Omega_0$ and $g_\sharp \mm=\mm$, where $g_\sharp \mm$ is the \emph{push forward of $\mm$ through $g$}. We recall that, if $X$ and $Y$ are Polish spaces, $f: X \to Y$ a Borel map and $\mu$ is a nonnegative and finite measure on $X$, then $f_\sharp\mu$ is defined by
\begin{equation}\label{eq:pushf}
\int_{Y} \varphi \de (f_{\sharp}\mu) = \int_X \varphi \circ f \de \mu 
\end{equation}
for every $\varphi:Y\to\R$ bounded (or nonnegative) Borel function.

If $\mathcal{A} \subset \cB$ is a sigma algebra on $\Omega$ we denote by $\rmS(\Omega, \cB, \mm; \mathcal{A})$ the subset of $\rmS(\Omega, \cB, \mm)$ of $\mathcal{A}-\mathcal{A}$ measurable maps. Finally $\symg{I_N}$ denotes the set of permutations of $I_N$ i.e.~bijective maps $\sigma: I_N \to I_N$.
\medskip

We consider the partial order on $\N$ given by
 \begin{equation}
   \text{$m\prec
 n\quad\Leftrightarrow\quad m\mid n$}\label{eq:166},
\end{equation}
where $m\mid n$ means that $n/m\in \N$. We write $m \precneq n$ if $m \prec n$ and $m \ne n$.

\begin{definition}[Segmentations]
\label{def:segm} Let $(\Omega, \cB)$ be a standard Borel space endowed with a nonatomic, positive finite measure $\mm$ and let $\cN \subset \N$ be an unbounded directed set w.r.t.~$\prec$. We say that a collection of partitions $(\mathfrak P_N)_{N \in \cN}$ of $\Omega$, with corresponding sigma algebras $\cB_N:= \sigma(\mathfrak P_N)$, is a \emph{$\cN$-segmentation of $(\Omega, \cB, \mm)$} if 
\begin{enumerate}
\item $\mathfrak P_N= \{\Omega_{N,k}\}_{k \in I_N}$ is a $N$-partition of $(\Omega, \cB)$ for every $N \in \cN$,
\item $\mm(\Omega_{N,k})=\mm(\Omega)/N$ for every $k\in I_N$ and every $N \in \cN$,
\item if $M\mid N=KM$ then
  $\bigcup_{k=0}^{K-1}\Omega_{N,mK+k}=\Omega_{M,m}$, $m\in
  I_{M}$,
\item $\sigma \left ( \left \{ \cB_N \mid N \in \cN \right \} \right) = \cB$.
\end{enumerate} 
In this case we call $(\Omega, \cB, \mm, (\mathfrak P_N)_{N \in \cN})$ a \emph{$\cN$-refined standard Borel measure space}.
\end{definition}

\begin{remark}
\label{rem:divideetimpera}
It is clear that $\cB_M \subset \cB_N$ if and only if 
$M \mid N$.
\end{remark}

\begin{example}\label{ex:canon} The canonical example of $\cN$-refined standard Borel measure space is 
\[ ([0,1), \mathcal{B}([0,1)), \lambda^c, (\mathfrak {I}_N)_{N \in \cN}),\]
where $\lambda^c$ is the one-dimensional Lebesgue measure restricted to $[0,1)$ and weighted by a constant $c>0$ and $\mathfrak {I}_N=(I_{N,k})_{k \in I_N}$ with $I_{N,k}:=[k/N,(k+1)/N)$, $k \in I_N$ and $N \in \cN$. 
\end{example}

\subsection{ Technical tools on standard Borel spaces and measure-preserving isomorphisms}\label{sec:laprima}

We start with the following fundamental result that follows by e.g. \cite[Theorem 9, Chapter 15]{royden}. 

\begin{theorem}[Isomorphisms of standard Borel spaces]
\label{theo:kura}
Let $(\Omega, \cB)$ and $(\Omega', \cB',)$ be standard Borel spaces endowed with nonatomic, positive finite measures $\mm$ and $\mm'$ respectively, such that $\mm(\Omega) = \mm'(\Omega')$. Then there exist two measurable functions $\varphi: \Omega \to \Omega'$ and $\psi: \Omega'\to \Omega$ such that 
\begin{equation}\label{eq:kura}
 \psi \circ \varphi = \ii_{\Omega} \text{ $\mm$-a.e.~in $\Omega$}, \quad \varphi \circ \psi = \ii_{\Omega'} \text{ $\mm'$-a.e.~in $\Omega'$}, \quad \varphi_\sharp \mm = \mm', \quad \psi_\sharp \mm'=\mm.  
\end{equation}
\end{theorem}
\begin{corollary}\label{cor:repre} Let $(\Omega, \cB)$ be a standard Borel space endowed with a nonatomic, positive finite measure $\mm$ and $(\Omega', \cB')$ be a standard Borel space. Then for every nonatomic, positive measure $\mu$ on $(\Omega', \cB')$ such that $\mu(\Omega')= \mm(\Omega)$, there exists a measurable map $X: \Omega \to \Omega'$ such that $X_\sharp \mm = \mu$.
\end{corollary}
\begin{lemma}[Existence of $\cN$-segmentations]
\label{le:exseg} For any standard Borel space $(\Omega, \cB)$ endowed with a nonatomic, positive finite measure $\mm$ and any unbounded directed set $\cN \subset \N$ w.r.t.~$\prec$, there exists a $\cN$-segmentation of $(\Omega, \cB, \mm)$.
\end{lemma}
\begin{proof} Let $([0,1), \mathcal{B}([0,1)), \lambda^c, (\mathfrak {I}_N)_{N \in \cN})$ be the $\cN$-refined standard Borel space of Example \ref{ex:canon} with $c=\mm(\Omega)$. Since $([0,1), \mathcal{B}([0,1)))$ is a standard Borel space endowed with the nonatomic, positive finite measure $\lambda^c$ such that $\mm(\Omega) = \lambda^c([0,1))$, by Theorem \ref{theo:kura} we can find measurable maps $\varphi:[0,1) \to \Omega$, $\psi: \Omega \to [0,1)$ and two subsets $\Omega_0 \in \cB$, $U \in \mathcal{B}([0,1))$ such that $\mm(\Omega_0)=\lambda^c(U)=0$, $\varphi\circ \psi= \ii_{\Omega \setminus \Omega_0}$, $\psi \circ \varphi = \ii_{[0,1)\setminus U}$, $\varphi_\sharp \lambda^c=\mm$ and $\psi_\sharp \mm=\lambda^c$. We can thus define 
\[ \Omega_{N,0} = \varphi(I_{N,0}\setminus U) \cup \Omega_0, \quad \Omega_{N,k} = \varphi(I_{N,k}\setminus U), \quad k \in I_N\setminus\{0\}, \, N \in \cN.\]
Setting $\mathfrak P_N := \{\Omega_{N,k}\}_{k \in I_N}$ for every $N \in \cN$, it is easy to check that $(\mathfrak P_N)_{N \in \cN}$ is a $\cN$-segmentation of $(\Omega, \cB, \mm)$.
\end{proof}

\begin{definition}[Compatible partitions]
If $(\Omega, \cB)$ and $(\Omega', \cB')$ are standard Borel spaces endowed with nonatomic, positive finite measures $\mm$ and $\mm'$ respectively, such that $\mm(\Omega)=\mm'(\Omega')$ and $\mathfrak P_N =\{\Omega_{N,k}\}_{k \in I_N}$ and $\mathfrak P'_N =\{\Omega'_{N,k}\}_{k \in I_N}$ are $N$-partitions of $(\Omega, \cB)$ and $(\Omega', \cB')$ respectively, we say that $\mathfrak P_N$ and $\mathfrak P'_N$ are \emph{$\mm-\mm'$ compatible} if
\[ \mm(\Omega_{N,k}) = \mm'(\Omega'_{N,k}) \quad \forall k \in I_N.\]
\end{definition}

\begin{lemma}[Isomorphisms preserving compatible partitions]
\label{le:mpm} Let $(\Omega, \cB)$ and $(\Omega', \cB')$ be standard Borel spaces endowed with nonatomic, positive finite measures $\mm$ and $\mm'$ respectively such that $\mm(\Omega)= \mm'(\Omega')$ and let $\mathfrak P_N=\{\Omega_{N,k}\}_{k \in I_N}$ and $\mathfrak P'_N=\{\Omega'_{N,k}\}_{k \in I_N}$ be two $\mm-\mm'$ compatible $N$-partitions of $(\Omega, \cB)$ and $(\Omega', \cB')$ respectively, for some $N \in \N$. Then there exist two functions $\varphi: \Omega \to \Omega'$ and $\psi: \Omega'\to \Omega$ such that
\begin{enumerate}
    \item $\varphi$ is $\cB$-$\cB'$ measurable and $\sigma(\mathfrak P_N)$-$\sigma(\mathfrak P'_N)$ measurable;
    \item $\psi$ is $\cB'$-$\cB$ measurable and $\sigma(\mathfrak P'_N)$-$\sigma(\mathfrak P_N)$ measurable;
    \item for every $k \in I_N$ it holds
    \begin{equation}\label{eq:subset}
\varphi(\Omega_{N,k}) \subset \Omega'_{N, k}, \quad \psi(\Omega'_{N,k}) \subset \Omega_{N, k};
\end{equation}
    \item for every $I \subset I_N$ it holds 
\begin{align*}
    \psi_I \circ \varphi_I &= \ii_{\Omega_I} \text{ $\mm_I$-a.e. in $\Omega_I$},\\
    \varphi_I \circ \psi_I &= \ii_{\Omega'_I} \text{ $\mm_I'$-a.e. in $\Omega_I'$},\\
    (\varphi_I)_\sharp \mm_I &= \mm_I',\\
    (\psi_I)_\sharp \mm_I'&=\mm_I,
\end{align*}
where the subscript $I$ denotes the restriction to $\cup_{k \in I} \Omega_{N,k}$ or $\cup_{k \in I} \Omega'_{N,k}$.
\end{enumerate}
\end{lemma}
\begin{proof} Applying Theorem \ref{theo:kura} to the standard Borel spaces $(\Omega_{\{k\}}, \cB_{\{k\}})$ and $(\Omega'_{\{k\}}, \cB'_{\{k\}})$ endowed, respectively, with the nonatomic, positive finite measures $\mm_{\{k\}}$ and $\mm'_{\{k\}}$ for every $k \in I_N$, we obtain the existence of measurable functions $\varphi_k, \psi_k$ satisfying \eqref{eq:kura} for each couple $\Omega_{N,k}$, $\Omega'_{N,k}$. It is then enough to define
\[ \varphi(\omega) := \varphi_k(\omega) \quad \text{ if } \omega \in \Omega_{N,k}, \quad \psi(\omega') := \psi_k(\omega') \quad \text{ if } \omega' \in \Omega'_{N,k}.\]
Notice that \eqref{eq:subset} is satisfied by construction.
\end{proof}

\begin{corollary}[Lifting permutations to isomorphisms]
\label{cor:isomor} Let $(\Omega, \cB, \mm)$ be a standard Borel space endowed with a nonatomic, positive finite measure $\mm$  and let $\mathfrak P_N=\{\Omega_{N,k}\}_{k \in I_N}$ be a $N$-partition of $(\Omega, \cB)$ for some $N \in \N$ such that $\mm(\Omega_{N,k})= \mm(\Omega)/N$ for every $k \in I_N$. If $\sigma \in \symg{I_N}$, there exists a measure-preserving isomorphism $g \in \rmS(\Omega, \cB, \mm; \sigma(\mathfrak P_N))$ such that 
\[ (g_k)_{\sharp}\mm|_{\Omega_{N,k}} = \mm|_{\Omega_{N,\sigma(k)}} \quad \forall k \in I_N, \]
where $g_k$ is the restriction of $g$ to $\Omega_{N,k}$.
\end{corollary}
\begin{proof} It is enough to apply Lemma \ref{le:mpm} to the standard Borel spaces $(\Omega, \cB)$ and $(\Omega', \cB')$=$(\Omega, \cB)$ endowed with the nonatomic, positive finite measures $\mm$ and $\mm'=\mm$, respectively, with the $N$-partitions $\mathfrak P_N$ and $\mathfrak P'_N=\{\Omega_{N,\sigma(k)}\}_{k \in I_N}$ respectively.
\end{proof}

\begin{corollary}\label{cor:mpsub} Let $(\Omega, \cB)$ be a standard Borel space endowed with a nonatomic, positive finite measure $\mm$ and let $\Omega_0, \Omega_1 \in \cB$ be such that $\mm(\Omega_0)=\mm(\Omega_1)>0$ and $\Omega_0 \cap \Omega_1 = \emptyset$. Then there exists a measure-preserving isomorphism $g \in \rmS(\Omega, \cB, \mm)$ such that
\[ (g_0)_\sharp \mm |_{\Omega_{0}} = \mm |_{\Omega_{1}}, \quad (g_1)_\sharp \mm |_{\Omega_{1}} = \mm |_{\Omega_{0}}, \quad g(\omega)= \omega  \text{ in } \Omega \setminus (\Omega_0 \cup \Omega_1),\]
where $g_i$ is the restriction of $g$ to $\Omega_k$, $k=0,1$.
\end{corollary}
\begin{proof} Applying Corollary \ref{cor:isomor} to the standard Borel space $(\Omega_0 \cup \Omega_1,\cB|_{\Omega_0 \cup \Omega_1})$ endowed with the nonatomic, positive finite measure $ \mm|_{\Omega_0 \cup \Omega_1}$  with the $2$-Borel partition $\mathfrak P_2= \{\Omega_k\}_{k=0,1}$ and $\sigma$ sending $0$ to $1$, we obtain the existence of a measure-preserving isomorphism $\tilde{g} \in \rmS(\Omega_0 \cup \Omega_1,\cB|_{\Omega_0 \cup \Omega_1}, \mm|_{\Omega_0 \cup \Omega_1})$ such that 
\[ (\tilde{g}_0)_\sharp \mm |_{\Omega_{0}} = \mm |_{\Omega_{1}}, \quad (\tilde{g}_1)_\sharp \mm |_{\Omega_{1}} = \mm |_{\Omega_{0}}, \]
where $\tilde{g}_i$ is the restriction of $\tilde{g}$ to $\Omega_k$, $k=0,1$. It is then enough to define $g: \Omega \to \Omega$ as
\[ g(\omega) = \begin{cases} \tilde{g}(\omega) \quad &\text{ if } \omega \in \Omega_0 \cup \Omega_1, \\ \omega \quad &\text{ if } \omega \in \Omega \setminus (\Omega_0 \cup \Omega_1). \end{cases}\]
\end{proof}

The next result is a particular case of Doob's Martingale Convergence Theorem for Banach-valued maps, 
see \cite[Theorem 6.1.12]{stroock}.
We recall that a filtration on $(\Omega, \cB)$ is a sequence $(\mathcal{F}_n)_{n \in \N}$ of sub-sigma algebras of $\cB$ such that $\mathcal{F}_n \subset \mathcal{F}_{n+1}$.
\begin{theorem}\label{theo:stroock}
Let $(\Omega, \cB)$ be a standard Borel space endowed with a nonatomic, positive finite measure $\mm$,  $(\mathcal{F}_n)_{n \in \N}$ be a filtration on $(\Omega, \cB)$ such that $ \sigma \left ( \left \{ \mathcal{F}_n \mid n \in \N \right  \} \right ) = \cB$, let 
 $\Y$ be a separable Banach space
and let $p\in [1,\infty).$
Then, given $X \in L^p(\Omega, \cB, \mm;\Y)$, the $\Y$-valued  martingale
\[ X_n:= \E_\mm \left [ X \mid \mathcal{F}_n \right ], \quad n \in \N, \]
satisfies
\begin{equation}\label{eq:lebdiffgu}
 \lim_{n \to + \infty}X_n  = X    
\end{equation}
both $\mm$-a.e.~and in $L^p(\Omega, \cB, \mm;\Y)$.
\end{theorem}

In general, the collection of sigma-algebras $(\mathcal{B}_N)_{N \in \cN}$  associated with a segmentation according to Definition \ref{def:segm} is not a filtration since it fails to be ordered by inclusion (recall Remark \ref{rem:divideetimpera}). 
However, it is always possible to extract from $(\mathcal{B}_N)_{N \in \cN}$ a filtration still satisfying item (4) in Definition \ref{def:segm}. More precisely we have the following result. 
\begin{lemma}[Cofinal filtrations]
\label{le:stroock}  Let $\cN \subset \N$ be an unbounded directed subset w.r.t.~$\prec$. Then there exists a totally ordered cofinal sequence $(b_n)_n \subset \cN$ satisfying
\begin{itemize}
    \item  $b_{n} \precneq b_{n+1}$ for every $n \in \N$,
    \item for every $N\in \N$ there exists $n\in \N$ such that $N\mid b_n.$
\end{itemize}
In particular,
for every $\cN$-refined standard Borel measure space $(\Omega, \cB, \mm, (\mathfrak P_N)_{N \in \cN})$ it holds that
$(\cB_{b_n})_{n \in \N}$ is a filtration on $(\Omega, \cB)$,
\begin{equation}\label{eq:goodseq}
\text{for every $N \in \cN$ there exists $n \in \N$ such that } \cB_N \subset \cB_{b_n},
\end{equation}
and 
$\sigma \left ( \left \{ \cB_{b_n} \mid n \in \N \right \} \right ) = \cB$.

For every $p \in [1,+\infty)$ and every separable Banach space $\Y$, we 
thus have that 
\begin{equation}
    \bigcup_{N \in \cN} L^{ p}(\Omega, \cB_N, \mm; \Y) \text{ is dense in } L^{ p}(\Omega, \cB, \mm; \Y).
\end{equation}
\end{lemma}
\begin{proof} Since $\cN$ is unbounded and directed, for every finite subset $\mathfrak M \subset \cN$ the quantity
\[ \text{succ}(\mathfrak M) := \min \left \{ N \in \cN \mid M \precneq N\, \forall M \in \mathfrak M \right \}\]
is well defined. Let $(a_n)_n \subset \N$ be an enumeration of $\cN$ and consider the following sequence defined by induction
\[ b_0=a_0, \quad b_{n+1}= \text{succ} \left (\{a_{n+1},b_n\} \right ), \quad n \in \N.\]
Then $b_{n} \precneq b_{n+1}$ for every $n \in \N$ and \eqref{eq:goodseq} holds for $(b_n)_n$ and any $\cN$-refined standard Borel measure space $(\Omega, \cB, \mm, (\mathfrak P_N)_{N \in \cN})$.
\end{proof}

In the next Lemma we show that, given two distinct points $\omega, \omega''$, they can always be separated by some partition $\mathfrak P_N$ for $N \in \cN$ sufficiently large.

\begin{lemma}[Separation property]
\label{le:separation}
Let $(\Omega, \cB, \mm, (\mathfrak P_N)_{N \in \cN})$ be a $\cN$-refined standard Borel measure space.
Then there exists $\Omega_0 \in \cB$ with $\mm(\Omega_0)=0$ such that for every $\omega', \omega'' \in \Omega \setminus \Omega_0$, $\omega' \ne \omega''$ there exists 
$M \in \cN$ such that for every $N\in \cN$, $M\mid N$ 
there exist  $k', k'' \in I_{N}$, $k' \ne k''$ with $\omega' \in \Omega_{N, k'}$ and $\omega'' \in \Omega_{N,k''}$.
\end{lemma}
\begin{proof} 
Let $(b_n)_n\subset \cN$ be a totally ordered cofinal sequence 
as in Lemma \ref{le:stroock} and let $\tau$ be a Polish topology on $\Omega$ such that $\cB$ coincides with the Borel sigma algebra generated by $\tau$. By \cite[Proposition 6.5.4]{Bogachev} there exists a countable family $\mathcal{F}$ of $\tau$-continuous functions $f: \Omega \to [0,1]$ separating the points of $\Omega$, meaning that for every $\omega', \omega'' \in \Omega$, $\omega' \ne \omega''$ there exists $f \in \mathcal{F}$ such that $f(\omega') \ne f(\omega'')$. Since $\mathcal{F} \subset L^2(\Omega, \cB, \mm; \R)$, by Theorem \ref{theo:stroock} with $\mathcal{F}_n:=\cB_{b_n}$, for every $f \in \mathcal{F}$ there exists a $\mm$-negligible set $\Omega_f$ such that
\[ \lim_{n \to + \infty} \E_\mm \left [ f \mid\sigma \left ( \mathfrak P_{b_n} \right ) \right ](\omega) = f(\omega) \quad \forall \omega \in \Omega \setminus \Omega_f.\]
Let $\Omega_0 := \cup_{f \in \mathcal{F}} \Omega_f$ and let $\omega', \omega'' \in \Omega \setminus \Omega_0$, $\omega' \ne \omega''$. We can find $f \in \mathcal{F}$ such that $f(\omega') \ne f(\omega'')$. 
Thus there exists $M=b_m \in \cN$ 
such that
\[ \E_\mm \left [ f \mid\sigma \left ( \mathfrak P_{M} \right ) \right ](\omega') \ne \E_\mm \left [ f \mid\sigma \left ( \mathfrak P_{M} \right ) \right ](\omega'').\]
Since $\E_\mm \left [ f \mid\sigma \left ( \mathfrak P_{M} \right ) \right ]$ is constant on every $\Omega_{M,k}$, $k \in I_{M}$, we conclude that the points $\omega'$ and $\omega''$ belong to different elements of $\mathfrak P_{M}$,
and therefore they also belong to different elements of 
$\mathfrak P_{N}$ for every $N\in \cN$ multiple of $M$.
\end{proof}

\begin{proposition}[Segmentation preserving isomorphisms]
\label{prop:isopotente}Let $(\Omega, \cB, \mm, (\mathfrak P_N)_{N \in \cN})$ and $(\Omega', \cB', \mm', (\mathfrak P'_N)_{N \in \cN})$ be $\cN$-refined standard Borel measure spaces such that $\mm(\Omega) = \mm'(\Omega')$. Then there exist two measurable functions $\varphi: \Omega \to \Omega'$ and $\psi: \Omega'\to \Omega$ such that for every $N \in \cN$ and every $I \subset I_N$ it holds
\[ \psi_I \circ \varphi_I = \ii_{\Omega_I} \text{ $\mm_I$-a.e. in $\Omega_I$}, \quad \varphi_I \circ \psi_I = \ii_{\Omega'_I} \text{ $\mm_I'$-a.e. in $\Omega_I'$}, \quad (\varphi_I)_\sharp \mm_I = \mm_I', \quad (\psi_I)_\sharp \mm_I'=\mm_I,\]
where the subscript $I$ denotes the restriction to $\cup_{k \in I} \Omega_{N,k}$ or $\cup_{k \in I} \Omega'_{N,k}$.
\end{proposition}
\begin{proof} By Lemma \ref{le:stroock}, it is enough to prove the statement in case $\cN = (b_n)_n$, where $(b_n)_n \subset \N$ is strictly $\prec$-increasing sequence and $(\Omega', \cB', \mm', (\mathfrak P'_N)_{N \in \cN})$ is $([0,1), \mathcal{B}([0,1)), \lambda^c, (\mathfrak{I}_N)_{N \in \cN})$ as in Example \ref{ex:canon} with $c= \mm(\Omega)$. By Lemma \ref{le:mpm}, we can find for every $n \in \N$ two measurable maps $\varphi_n: \Omega \to [0,1)$ and $\psi_n:[0,1) \to \Omega$ satisfying the thesis of Lemma \ref{le:mpm} for the standard Borel spaces $(\Omega, \cB)$ and $([0,1), \mathcal{B}([0,1))$ endowed with nonatomic, positive and finite measures $\mm$ and $\lambda^c$ respectively and the $\mm-\lambda^c$ compatible $b_n$-partitions of $(\Omega, \cB)$ and $([0,1), \mathcal{B}([0,1)))$ given by $\mathfrak P_{b_n}$ and $\mathfrak I_{b_n}$, where we recall from Example \ref{ex:canon} that $\mathfrak I_{b_n}= (I_{b_n,k})_{k \in I_{b_n}}$ with $I_{b_n,k}=[k/b_n, (k+1)/b_n)$. Since $\sum_n b_n^{-1} < + \infty$, for every $\omega \in \Omega$ the sequence $(\varphi_n(\omega))_n \subset [0,1)$ is Cauchy, hence converges. We thus have the existence of a measurable map $\varphi: \Omega \to [0,1)$ such that
\[ \varphi(\omega) = \lim_n \varphi_n (\omega) \quad \forall \omega \in \Omega.\]
If $n \in \N$, $k \in I_{b_n}$ and $ \xi \in \rmC_b(I_{b_n,k})$ then
\begin{align*}
 \int_{I_{b_n, k}} \xi  \de \varphi_\sharp \mm &= \int_{\Omega_{b_n,k}}  \xi(\varphi(\omega)) \de \mm(\omega) = \lim_m \int_{\Omega_{b_n,k}}  \xi(\varphi_m(\omega)) \de \mm(\omega) \\
& = \lim_m \int_{I_{b_n,k}}  \xi \de \lambda^c = \int_{I_{b_n,k}}  \xi\de \lambda^c,
\end{align*}
since for $m$ sufficiently large $(\varphi_m)_\sharp \mm|_{\Omega_{b_n,k}}=\lambda^c|_{I_{b_n,k}}$ by Lemma \ref{le:mpm}. This shows that $\varphi_\sharp \mm |_{\Omega_{b_n,k}} = \lambda^c|_{I_{b_n,k}}$ for every $k \in I_{b_n}$ and every $n \in \N$. To conclude it is enough to show that $\varphi$ is $\mm$-essentially injective. Let $\Omega_0 \subset \Omega$ be the $\mm$-negligible subset of $\Omega$ given by Lemma \ref{le:separation} and let $\Omega_1:= \varphi^{-1}(J)$, where
\[ J := \left \{ k/b_n \mid k \in I_{b_n}, \, n \in \N \right \} \subset [0,1).\]
Since $\lambda^c(J)=0$, then $\mm(\Omega_1)=0$; let $\omega', \omega'' \in \Omega \setminus (\Omega_0 \cup \Omega_1)$. Then there exists $M \in \N$ such that $\omega'$ and $\omega''$ belong to different elements of $\mathfrak P_{b_n}$ for every $n \ge M$. By \eqref{eq:subset} and Lemma \ref{le:separation}, we can find $k',k'' \in I_{b_M}$ with $k \ne k'$ such that, $\varphi_n(\omega') \in I_{b_{M}, k'}$ and $\varphi_n(\omega'') \in I_{b_{M}, k''}$ for every $n \ge M$. Thus $\varphi(\omega') \in \overline{I_{b_{M}, k'}}$ and $\varphi(\omega') \in \overline{I_{b_{M}, k''}}$; however, since
\[ \overline{I_{b_{M}, k'}} \cap \overline{I_{b_{M}, k''}} \subset J,\]
it must be that $\varphi(\omega') \ne \varphi(\omega'').$
\end{proof}
\subsection{Approximation of couplings by using measure-preserving isomorphisms}\label{sec:laseconda}
 If $X$ is a Polish space, we denote by $\prob(X)$ the space of Borel probability measures on $X$ which is endowed with the weak (or narrow) topology: a sequence $(\mu_n)_n \subset \prob(X)$ converges to $\mu \in \prob(X)$ if 
\[ \lim_n \int_X \varphi \de \mu_n = \int_X \varphi \de \mu \]
for every $\varphi:X \to \R$ continuous and bounded. In this case, we write $\mu_n \to \mu$ in $\prob(X)$. 

If $X,Y$ are Polish spaces and $(\mu,\nu)\in\prob(X)\times\prob(Y)$, we define the set of admissible transport plans
\begin{equation}\label{def:admplans}
 \Gamma(\mu, \nu) := \left  \{ \ggamma \in \prob(X \times Y) \mid \pi^{1}_{\sharp} \ggamma = \mu \, , \, \pi^{2}_{\sharp} \ggamma = \nu \right \},
\end{equation}
where $\pi^i$, $i=1,2$, denotes the projection on the $i$-th component and we call $\pi^{i}_{\sharp}\ggamma$ the $i$-th marginal of $\ggamma$.

\begin{definition}[Wasserstein spaces]
Let $\X$ be a separable Banach space,  $\mu\in\prob(\X)$, $p\ge 1$. We define the space
\begin{equation}\label{condTanTX}
\prob_p(\X) := \{ \mu \in \prob(\X) \mid \int_\X |x|^p \de \mu(x) < + \infty \}.
\end{equation}
Given $\mu,\nu\in\prob_p(\X)$, we define the $L^p$-Wasserstein distance $W_p$ by
\begin{align} \label{eq:w21} W_p^p(\mu, \nu) &:= \inf \left \{ \int_{\X \times \X} |x-y|^p \de \ggamma(x,y) \mid \ggamma \in \Gamma(\mu, \nu) \right \}.
\end{align}
We denote by $\Gamma_o(\mu, \nu)$ the (non-empty, compact and convex) subset of admissible plans in $\Gamma(\mu, \nu)$ realizing the infimum in
\eqref{eq:w21}.
\end{definition}

We recall that $(\prob_p(\X), W_p)$ is a complete and separable metric space. Moreover, if $(\mu_n)_{n\in\N}\subset\prob_p(\X)$ and $\mu\in\prob_p(\X)$, the following holds (see \cite[Prop. 7.1.5,  Lem. 5.1.7]{ags})
\begin{equation}
  \label{eq:important}
  \mu_n\to\mu\text{ in }\prob_p(\X),\text{ as }n\to+\infty \quad\Longleftrightarrow\quad\begin{cases}\mu_n\to\mu \text{ in }\prob(\X),\\
    \int_\X |x|^p \de \mu_n\to\int_\X |x|^p \de \mu,
  \end{cases}
  \text{ as }n\to+\infty.\\
\end{equation}
We refer e.g.~to \cite[Chapter 7]{ags} for a more comprehensive introduction to Wasserstein distances. 

The following result is an application of \cite[Theorem
1.1]{Brenier-Gangbo03}.
We will use the following notation: if $\X_1$ and $\X_2$ are sets and $X_1:\X_1\to \Y_1$, $X_2:\X_2\to\Y_2$ we denote by $X_1\otimes
X_2:\X_1\times \X_2\to \Y_1\times \Y_2$ the
map
$(x_1,x_2)\mapsto (X_1(x_1),X_2(x_2))$. 
\begin{theorem}[Approximation of bistochastic couplings]
\label{thm:gangbo}
Let $(\Omega, \cB, \P, (\mathfrak P_N)_{N \in \cN})$ be a $\cN$-refined standard Borel probability space.
Then for every $\ggamma \in \Gamma(\P, \P)$ there exist a totally ordered strictly increasing sequence $(N_n)_n \subset \cN$ and maps $g_n \in \rmS(\Omega, \cB, \P; \cB_{N_n})$ such that, for every separable Banach spaces  $\mathsf{Z}, \mathsf{Z}'$ and every $Z \in L^0(\Omega, \cB, \P; \mathsf{Z})$, $Z' \in L^0(\Omega, \cB, \P; \mathsf{Z}')$, it holds
\begin{equation}\label{eq:gangboconv}
  (Z \otimes Z')_\sharp (\ii_{\Omega}, g_n)_\sharp \P \to
 (Z\otimes Z')_\sharp \ggamma
   \text{ in } \prob(\mathsf{Z}\times \mathsf{Z}').
\end{equation}
\end{theorem}
\begin{proof} 
By Lemma \ref{le:stroock} it is not restrictive to assume that $\cN=(b_n)_n$ for a totally ordered strictly increasing sequence $(b_n)_n$, $n\in \N$. We divide the proof in several steps.

\smallskip \noindent
(1) \emph{Let $([0,1), \mathcal{B}([0,1)), \lambda^1, (\mathfrak {I}_N)_{N \in \cN})$ be the $\cN$-refined standard Borel probability space of Example \ref{ex:canon} with $c=1$. Then for every $\ggamma \in \Gamma(\lambda^1, \lambda^1)$, there exist a strictly increasing sequence $(N_n)_n \subset \N$ and maps
$g_n \in \rmS([0,1), \mathcal{B}([0,1)), \lambda^1; \sigma(\mathfrak {I}_{b_{N_n}}))$ such that
\[ (\ii_{[0,1)}, g_n)_\sharp \lambda^1 \to \ggamma \text{ in } \prob([0,1) \times [0,1)).\]}%
Let $\bar{\mathcal{L}}$ be the one dimensional Lebesgue measure restricted to $[0,1]$ and let $\ggamma \in \Gamma(\lambda^1, \lambda^1)$. Let $\mmu \in \prob([0,1] \times [0,1]) $ be an extension of $\ggamma$ to $[0,1] \times [0,1]$ such that $\mmu \in \Gamma(\bar{\mathcal{L}}, \bar{\mathcal{L}})$. In \cite[Theorem 1.1]{Brenier-Gangbo03} it is proven that it is possible to find a strictly increasing sequence $(N_n)_n \subset \N$ and maps $(f_n)_n \subset \rmS([0,1], \mathcal{B}([0,1]), \bar{\mathcal{L}})$ such that for every $n \in \N$ there exists $\sigma_n \in \symg{I_{2^{N_n}}}$ such that
\begin{equation}\label{eq:tsgangbo1}
f_n(x) = x - x_{N_n, k} + x_{N_n, \sigma_n(k)}, \quad x \in I_{2^{N_n}, k}, \quad k \in I_{2^{N_n}},
\end{equation}
with $x_{m,j}$ being the center of $I_{2^m, j}$, and
satisfying 
\begin{equation}\label{eq:tsgangbo2}
(\ii_{[0,1]}, f_n)_\sharp \bar{\mathcal{L}} \to \mmu \text{ in } \prob([0,1]\times [0,1]).
\end{equation}
If we call $g_n$ the restriction of $f_n$ to $[0,1)$, $n \in \N$, we get that $g_n \in \rmS([0,1), \mathcal{B}([0,1)), \lambda^1; \sigma(\mathfrak {I}_{b_{N_n}}))$ for every $n \in \N$ and
\[
(\ii_{[0,1)}, g_n)_\sharp \lambda^1 \to \ggamma \text{ in } \prob([0,1)\times [0,1)). 
\]
This proves the first step only in case $b_n=2^n$. However, it can be easily checked that the proof of \cite[Theorem 1.1]{Brenier-Gangbo03} does not depend on the specific choice of the sequence $b_n$ but it is enough that $b_n \precneq b_{n+1}$ for every $n \in \N$ so that the length of the interval $[k/b_n, (k+1)/b_n]$ goes to $0$ faster than $2^{-n}$ as $n \to + \infty$. This concludes the proof of the first claim.

\medskip \noindent
(2) \emph{Let $([0,1), \mathcal{B}([0,1)), \lambda^1, (\mathfrak {I}_N)_{N \in \cN})$ be the $\cN$-refined standard Borel probability space of Example \ref{ex:canon} with $c=1$. Then for every $\ggamma \in \Gamma(\lambda^1, \lambda^1)$, there exist a strictly increasing sequence $(N_n)_n \subset \N$ and maps $g_n \in \rmS([0,1), \mathcal{B}([0,1)), \lambda^1; \sigma(\mathfrak {I}_{b_{N_n}}))$ such that, for every separable Banach spaces  $\mathsf{Z}, \mathsf{Z}'$ and every $Z \in L^0([0,1), \mathcal{B}([0,1)), \lambda^1; \mathsf{Z})$, $Z' \in L^0([0,1), \mathcal{B}([0,1)), \lambda^1; \mathsf{Z}')$, it holds
  \[ (Z \otimes Z')_\sharp (\ii_{[0,1)}, g_n)_\sharp \lambda^1 \to
     (Z\otimes Z')_\sharp \ggamma
     \text{ in } \prob(\mathsf{Z}\times \mathsf{Z}').\]}%
Let $\ggamma \in \Gamma(\lambda^1, \lambda^1)$ and let $(g_n)_n$ be
the sequence given by claim (1) for $\ggamma$. Let $\mathsf{Z}$ and $\mathsf{Z}'$ be separable Banach spaces and let $Z \in L^0([0,1), \mathcal{B}([0,1)), \lambda^1; \mathsf{Z})$, $Z' \in L^0([0,1), \mathcal{B}([0,1)), \lambda^1; \mathsf{Z}')$. Observe that for every $\eps>0$
there exists a compact set $K_\eps\subset [0,1)$ such that the
restrictions of $Z$ and $Z'$ to $K_\eps$ are continuous in $K_\eps$ and
$\lambda^1([0,1)\setminus K_\eps)<\eps$, so that, setting
$\gamma_n:=(\ii_{[0,1)}, g_n)_\sharp \lambda^1$, $n \in \N$, we have
that $\gamma_n([0,1)^2\setminus K_\eps^2)\le 2\eps$ for every $n\in
\N$. By \cite[Proposition 5.1.10]{ags} and claim (1), $(Z\otimes Z')_\sharp
(\ii_{[0,1)}, g_n)_\sharp \lambda^1 \to
 (Z\otimes Z')_\sharp \ggamma
\text{ in } \prob(\mathsf{Z}\times \mathsf{Z}')$.

\medskip \noindent
(3) \emph{Conclusion} Let $\ggamma \in \Gamma(\P, \P)$ and let
$\varphi:\Omega \to [0,1)$ and $\psi: [0,1) \to \Omega$ be the maps
given by Proposition \ref{prop:isopotente} for the $\cN$-refined
standard Borel probability spaces $(\Omega, \cB, \P, (\mathfrak
P_N)_{N \in \cN})$ and $([0,1), \mathcal{B}([0,1)), \lambda^1,
(\mathfrak {I}_N)_{N \in \cN})$, where the latter is as in Example
\ref{ex:canon} with $c=1$. If we define $\ggamma':=(\varphi
\otimes
\varphi)_\sharp \ggamma$, we have that $\ggamma' \in \Gamma(\lambda^1, \lambda^1)$ so that we can find a strictly increasing sequence $(N_n)_n \subset \N$ and maps $g'_n \in \rmS([0,1), \mathcal{B}([0,1)), \lambda^1; \sigma(\mathfrak {I}_{b_{N_n}}))$ as in step (2). Let us define 
\[ g_n:= \psi \circ g'_n \circ \varphi, \quad n \in \N.\]
Then, up to change each $g_n$ on a $\P$-negligible set of points, we can assume that $g_n \in \rmS(\Omega, \cB, \P; \cB_{b_{N_n}})$. Let $\mathsf{Z}$ and $\mathsf{Z}'$ be separable Banach spaces and let $Z \in L^0(\Omega, \cB, \P; \mathsf{Z})$, $Z' \in L^0(\Omega, \cB, \P; \mathsf{Z}')$. If we define  $Z_0:=Z \circ \psi$ and $Z'_0:=Z' \circ \psi$,  we get that $Z_0 \in L^0([0,1), \mathcal{B}([0,1)), \lambda^1); \mathsf{Z})$, $Z_0' \in L^0([0,1), \mathcal{B}([0,1)), \lambda^1); \mathsf{Z}')$. By step (2) we thus get
\[ (Z_0\otimes Z_0')_\sharp (\ii_{[0,1)}, g'_n)_\sharp \lambda^1 \to (Z_0\otimes Z_0')_\sharp \ggamma' \text{ in } \prob(\mathsf{Z}\times \mathsf{Z}') \]
which is equivalent to \eqref{eq:gangboconv}.
\end{proof}

\begin{remark}\label{rem:sepa} In the setting of Theorem \ref{thm:gangbo}, let  $\varphi:\mathsf{Z}\to [0,+\infty)$, $\varphi':\mathsf{Z}'\to [0,+\infty)$ be Borel functions.

Setting 
\[ \psi(z,z')=\varphi(z)+\varphi'(z'), \quad \text{ for every } (z,z') \in \mathsf{Z} \times \mathsf{Z}',\]
then we have
for every $n\in \N$
\begin{align*}\label{eq:momZ}
    \int \psi \de (Z \otimes Z')_\sharp (\ii_{\Omega}, g_n)_\sharp \P &=
    \int \psi(Z,Z'\circ g_n)\,\d
    \P=
    \int 
    \Big(\varphi(Z)+\varphi'(Z'\circ g_n)\Big)\,\,\d
     \P
    \\&=
    \int \varphi(Z)\,\d\P+
    \int \varphi'(Z')\,\d\P
    \\&=
    \int \varphi(Z(x_1))\,\d\ggamma(x_1,x_2)+
    \int \varphi'(Z'(x_2))\,\d\ggamma(x_1,x_2)
    \\&=
    \int \psi(Z(x_1),Z'(x_2))\,
    \d\ggamma(x_1,x_2)=
    \int \psi \de (Z\otimes Z')_\sharp \ggamma.  
\end{align*} 
As a consequence, if $\mathsf{Z}=\mathsf{Z}'=\X$ for a separable Banach space $\X$ and $Z,Z' \in L^p(\Omega, \cB, \P; \X)$, $p \in [1,+\infty)$, then the convergence in \eqref{eq:gangboconv} holds in $\prob_p(\X^2)$. To prove this, it sufficies to apply \eqref{eq:important} and choose $\psi(z,z'):= |z|_{\X}^p + |z'|_{\X}^p$, $z,z' \in \X$, in
the above identity.
\end{remark}

We deduce two important applications.
\begin{corollary}
\label{cor:marc} Let $(\Omega, \cB, \P, (\mathfrak P_N)_{N \in \cN})$ be a $\cN$-refined standard Borel probability space.
Then for every $\ggamma \in \Gamma(\P, \P)$ there exist a totally ordered strictly increasing sequence $(N_n)_n \subset \cN$ and maps
$g_n \in \rmS(\Omega, \cB, \P; \cB_{{N_n}})$ such that, for every Polish topology $\tau$ on $\Omega$ generating $\cB$, it holds
\[ (\ii_{\Omega}, g_n)_\sharp \P \to \ggamma \text{ in } \prob(\Omega \times \Omega, \tau \otimes \tau),\]
where $\tau \otimes \tau$ is the product topology on $\Omega \times \Omega$.
\end{corollary}
\begin{proof} By Theorem \ref{thm:gangbo} and Remark \ref{rem:sepa}, we have the existence of a strictly increasing sequence $(N_n)_n \subset \N$ and maps $g_n \in \rmS(\Omega, \cB, \P; \cB_{{N_n}})$ such that, choosing the separable Hilbert space $\R$, we get
\[
(\varphi_1\otimes \varphi_2)_\sharp (\ii_{\Omega}, g_n)_\sharp \P \to (\varphi_1\otimes \varphi_2)_\sharp \ggamma \text{ in } \prob_2(\R^2)
\]
for every $\varphi_1, \varphi_2 \in \rmC_b(\Omega, \tau) \subset L^2(\Omega,\cB, \P; \R)$. Since the range of $\varphi_i$ is bounded and thus relatively compact, 
by the $\prob(\R^2)$ convergence we get that 
\[ \int_{\Omega \times \Omega} h(\varphi_1(\omega_1), \varphi_2(\omega_2)) \de \ggamma_n(\omega_1, \omega_2) \to \int_{\Omega \times \Omega} h(\varphi_1(\omega_1), \varphi_2(\omega_2)) \de \ggamma(\omega_1, \omega_2)\]
for every continuous function $h: \R^2 \to \R$ where $\ggamma_n = (\ii_{\Omega}, g_n)_\sharp \P$, $n \in \N$. Choosing $h(x,y):= xy$, we get that
\begin{equation}\label{eq:algebra}
 \int_{\Omega \times \Omega} \varphi_1(\omega_1) \varphi_2(\omega_2) \de \ggamma_n(\omega_1, \omega_2) \to \int_{\Omega \times \Omega} \varphi_1(\omega_1) \varphi_2(\omega_2) \de \ggamma(\omega_1, \omega_2) \quad \forall \varphi_1, \varphi_2 \in \rmC_b(\Omega, \tau).
\end{equation}
Let $\mathcal{A} \subset \rmC_b(\Omega, \tau)$ be a unital subalgebra whose induced initial topology on $\Omega$ coincides with $\tau$ (e.g.~the subset of $\mathsf d$-Lipschitz continuous and bounded functions for a complete distance $\mathsf d$ inducing $\tau$). It is easy to check that
\[ \mathcal{A} \otimes \mathcal{A} := \left \{ \sum_{i=1}^n \varphi_1^i \otimes \varphi_2^i \mid (\varphi_1^i)_{i=1}^n, (\varphi_2^i)_{i=1}^n \subset \mathcal{A}, \, n \in \N \right \} \subset \rmC_b(\Omega \times \Omega, \tau \otimes \tau) \]
is a unital subalgebra whose induced initial topology on $\Omega \times \Omega$ coincides with $\tau \otimes \tau$. By \eqref{eq:algebra} we thus have that
\[ \int_{\Omega \times \Omega} \varphi \de \ggamma_n \to \int_{\Omega \times \Omega} \varphi \de \ggamma \quad \forall \varphi \in \mathcal{A} \otimes \mathcal{A}.\]
We conclude by \cite[Lemma 2.3]{SS20}.
\end{proof}

The second part of the following corollary represents a sort of extension of the known result in \cite[Lemma 6.4]{carda} (cf. also \cite[Lemma 5.23 p. 379]{CD18}) to the class of pairs $(X,Y)$ of random variables, where $X$ and $Y$ take values on possibly different separable Banach spaces, with possibly different $p$-integrability. Whenever the joint distribution of two pairs $(X,Y),(X',Y')$ belonging to such a class is equal, we are able to prove the existence of a sequence of measure-preserving maps giving the desired strong approximation result for both the components.
\begin{corollary}\label{cor:from-gangbo}
Let $(\Omega, \cB, \P, (\mathfrak P_N)_{N \in \cN})$ be a
$\cN$-refined standard Borel probability space, 
let $\mathsf{Z}$ and $\mathsf{Z}'$ be separable Banach spaces and let $Z \in L^0(\Omega, \cB, \P; \mathsf{Z})$, $Z' \in L^0(\Omega, \cB, \P; \mathsf{Z}')$. Then for every $\boldsymbol\mu\in \Gamma(Z_\sharp \P,Z'_\sharp \P)$ there exist a totally ordered strictly increasing sequence $(N_n)_n \subset \cN$ and maps
$g_n \in \rmS(\Omega, \cB, \P; \cB_{{N_n}})$ such that
\begin{equation}\label{eq:gangboconv2}
(Z,Z'\circ g_n)_\sharp 
\P \to \boldsymbol \mu \text{ in } \prob(\mathsf{Z}\times \mathsf{Z}').
\end{equation}
In particular, if $\X$ and $\mathsf{Y}$ are separable Banach spaces, $X,X' \in L^p(\Omega, \cB, \P; \mathsf{X})$, $Y,Y' \in L^q(\Omega, \cB, \P; \mathsf{Y})$, $p,q \in [1,+\infty)$, and $(X,Y)_\sharp \P=(X',Y')_\sharp \P$, then there exist a  totally ordered strictly increasing sequence $(N_n)_n \subset \cN$ and maps
$g_n \in \rmS(\Omega, \cB, \P; \cB_{{N_n}})$ such that
$X'\circ g_n\to X$ in $L^p(\Omega, \cB, \P; \X)$ and $Y'\circ g_n\to Y$ in $L^q(\Omega, \cB, \P; \mathsf{Y})$ as $n\to\infty$.
\end{corollary}
\begin{proof} Let $\mu:=Z_\sharp \P$ and $\mu':=Z'_\sharp \P$;
  let us first observe that there exists $\ggamma\in \Gamma(\P,\P)$
  such that $(Z\otimes Z')_\sharp \ggamma=\mmu$.
  In fact, we can disintegrate $\P$ w.r.t.~$Z$ and $\mu$ (see e.g. \cite[Theorem 5.3.1]{ags} for details on the disintegration theorem), obtaining
  a Borel family of measures $(\theta_z)_{z\in \mathsf{Z}}\subset \mathcal P(\Omega)$
  such that $\theta_z$ is concentrated on $Z^{-1}(z)$ for
  $\mu$-a.e.~$z \in \mathsf{Z}$ and
  $\P=\int \theta_z \,\d\mu(z)$; similarly, we can also find
  $(\theta'_{z'})_{z'\in \mathsf{Z}'}\subset \mathcal P(\Omega)$ such that
  $\theta_{z'}'$ is concentrated on $(Z')^{-1}(z')$ for $\mu'$-a.e.~$z' \in \mathsf{Z}'$ and
  $\P = \int \theta'_{z'}\,\d\mu'(z')$. We can thus define
  $\ggamma:=\int (\theta_{z}\otimes \theta'_{z'})\,\d\boldsymbol
  \mmu(z,z')\in \mathcal P(\Omega\times \Omega)$ and it is immediate to
  check that
  $(Z\otimes Z')_\sharp\ggamma=\mmu$ since for every function
  $\varphi_1\in \mathrm C_b(\mathsf{Z})$, $\varphi_2\in \mathrm C_b(\mathsf{Z'})$,
  \begin{align*}
    \int
    \varphi_1(Z(\omega))\varphi_2(Z'(\omega'))\,\d\ggamma(\omega,\omega')
    &=
    \int \bigg(\int
    \varphi_1(Z(\omega))\varphi_2(Z'(\omega'))\,\d\theta_{z}(\omega)\,\d\theta'_{z'}(\omega')
      \bigg)\,\d
      \boldsymbol\mu(z,z')
    \\&=
    \int \int
    \varphi_1(Z(\omega))\,\d\theta_z(\omega)
    \int\varphi_2(Z'(\omega'))\,\d\theta'_{z'}(\omega')
    \,\d
    \boldsymbol\mu(z,z')
    \\&=
    \int \Big(
    \varphi_1(z)
    \varphi_2(z')\Big)
    \,\d
    \boldsymbol\mu(z,z').
  \end{align*}
  Notice also that it is enough to verify that $(Z,Z')_\sharp \P$ and $\ggamma$ have the same integral for every function of the form $\varphi_1 \otimes \varphi_2$ as above to conclude that the two measures coincide (see e.g.~the proof of the above Corollary \ref{cor:marc}).
  We can then apply \eqref{eq:gangboconv} and obtain
  \eqref{eq:gangboconv2}.
  
  Let us show the last part of the statement: we take $\mathsf{Z}=\mathsf{Z}':=\X \times \mathsf{Y}$ and, by \eqref{eq:gangboconv2} with $Z:=(X,Y)$, $Z':=(X',Y')$ and $\mmu:=(\ii_{\X \times \mathsf{Y}},\ii_{\X \times \mathsf{Y}})_\sharp (X,Y)_\sharp \P$, we have that 
\[
 (X,Y,X'\circ g_n,Y'\circ g_n)_\sharp 
\P \to \boldsymbol \mu \text{ in }  \prob((\X \times \mathsf{Y})^2).
\]
We thus have that $(X,X'\circ g_n)_\sharp \P \to (\ii_\X,\ii_\X)_\sharp X_\sharp \P$ in $\prob(\X^2)$. By Remark \ref{rem:sepa} with $\psi((x,y),(x',y'))= |x|_\X^p + |x'|_\X^p$, $x,x \in \X, y,y' \in \mathsf{Y}$, we get, also using \eqref{eq:important}, that $(X,X'\circ g_n)_\sharp \P \to (\ii_\X,\ii_\X)_\sharp X_\sharp \P$ in $\prob_p(\X^2)$. As a consequence (see e.g.\cite[Proposition 7.1.5, Lemma 5.1.7]{ags}), we get
  \begin{align*}
    \lim_{n\to\infty} \int |X-X'\circ g_n|^p\,\d\P
    &=
    \lim_{n\to\infty} \int |x-x'|^p\,\d\big((X,X'\circ g_n)_\sharp
      \P)(x,x')\\
      &=
      \int |x-x'|^p\,\d((\ii_\X,\ii_\X)_\sharp X_\sharp \P)(x,x')\\
      &=0.
  \end{align*}
The proof for $Y$ and $Y'$ is identical.
\end{proof}

\begin{remark}\label{rem:sepa2} In the same setting of Corollary \ref{cor:from-gangbo} and similarly to Remark \ref{rem:sepa}, if $\mathsf{Z}=\mathsf{Z}'=\X$ for a separable Banach space $\X$ and $X,X' \in L^p(\Omega, \cB, \P; \X)$, $p \in [1,+\infty)$, then \eqref{eq:important} gives that the convergence in \eqref{eq:gangboconv2} holds in $\prob_p(\X^2)$,    
\end{remark}

As a byproduct, we recover the following important 
result 
(see e.g.~\cite[Lemma 3.13]{gangbotudo} for a statement in case $p=2$ and $\X=\R^d$), 
which is also related
to the equivalence 
between the Monge and the Kantorovich formulations of Optimal Transport problems 
\cite[Theorem 2.1, Theorem 9.3]{lnamb},  \cite[Theorem B]{pra}.

\begin{proposition} \label{prop:pratelli} Let $(\Omega, \cB)$ be a standard Borel space endowed with a nonatomic probability measure $\P$, let $\X$ be a separable Banach space and let $p\in[1,+\infty)$. If $\mu, \nu \in \prob_p(\X)$ and $X \in L^p(\Omega, \cB,\P; \X)$ is s.t.~$X_\sharp \P = \mu$, then, for every $\eps>0$, there exists $Y \in L^p(\Omega, \cB,\P; \X)$ s.t. $Y_\sharp \P = \nu$ and
\[ |X-Y|_{L^p(\Omega \cB,\P; \X)} \le W_p(\mu, \nu) + \eps.\]
\end{proposition}
\begin{proof}
Let us consider the $\cN$-refined standard Borel probability space $(\Omega, \cB, \P, (\mathfrak P_N)_{N \in \cN})$ with $\cN =(2^k)_{k \in \N}$; let $\mmu \in \Gamma_o(\mu, \nu)$ and let $X'\in L^p(\Omega,\cB,\P; \X)$ be such that $X'_\sharp\P=\nu$. By Corollary \ref{cor:from-gangbo} and Remark \ref{rem:sepa2}, there exist a strictly increasing sequence $(N_n)_n \subset \cN$ and maps
$g_n \in \rmS(\Omega, \cB, \P; \cB_{{N_n}})$ such that
\begin{equation*}
(X,X'\circ g_n)_\sharp 
\P \to \boldsymbol \mu \text{ in } \prob_p(\X^2).
\end{equation*}
We have
\begin{align*}
    \lim_{n\to\infty} \int |X-X'\circ g_n|^p\,\d\P
    &=
    \lim_{n\to\infty} \int |x-y|^p\,\d\big((X,X'\circ g_n)_\sharp
      \P)(x,y)\\&=\int |x-y|^p\,\d\mmu(x,y)
      =W_p^p(\mu,\nu)
      \end{align*}
Thus given $\eps>0$ it is always possible to find $n \in \N$ sufficiently large such that $Y:=X' \circ g_n$ satisfies the thesis.      
\end{proof}

\section{Monotone-dissipative operators 
and Lipschitz maps in \texorpdfstring{$L^p$}{L}-spaces invariant by measure-preserving 
transformations}
\label{subsec:invariant-maps}
Let $(\Omega, \cB)$ be a standard Borel space endowed with a nonatomic probability measure $\P$ (see Definition \ref{def:sbs}). We denote by $\rmS(\Omega)$ the class of
$\cB$-$\cB$-measurable maps $g:\Omega\to\Omega$ which are
essentially injective and measure-preserving, i.e. there exists a full $\P$-measure set $\Omega_0 \in \cB$ such that $g$ is injective on $\Omega_0$ and $g_\sharp \P=\P$.
If $g\in \rmS(\Omega)$, there exists $g^{-1}\in
\rmS(\Omega)$
(defined up to a $\P$-negligible set) such that $g^{-1}\circ g=g\circ
g^{-1}=\ii_\Omega$ $\P$-a.e.~in $\Omega$.

 Consider two separable Banach spaces $\Y,\Z$ and fix exponents $p,q\in [1,+\infty)$. We set
\begin{equation}
\label{eq:XY-notation}
\cY:=L^p(\Omega, \cB, \P;\Y),\quad
\cZ:=L^q(\Omega, \cB, \P;\Z).
\end{equation}
Notice that 
for every $g\in \rmS(\Omega)$ the 
pullback transformation
$g^*:\cY\mapsto \cY$ sending $X$ to $X \circ g$
is a linear isometry of $\cY$: in particular
\begin{align}
    X_n\to X\text{ strongly in $\cY$}
    \quad&\Longrightarrow
    \quad
    g^* X_n\to g^* X\text{ strongly in $\cY$},\\
    X_n\to X\text{ weakly in $\cY$}
    \quad&\Longrightarrow
    \quad
    g^* X_n\to g^* X\text{ weakly in $\cY$}.
\end{align}

The aim of this section is to study properties of maps and sets/operators, defined on these particular spaces, which are invariant by measure-preserving transformations. We will also apply the results of Section \ref{sec:appC} to this particular setting. The interest on such kind of properties is made evident by the implications in the study of dissipative evolutions in Wasserstein spaces via Lagrangian representations (cf. \cite{CSS2grande}).

\begin{definition}[Invariant sets and maps]
\label{def:inv}
We say that a set $\Bb\subset\cY\times\cZ$
is \emph{invariant by measure-preserving isomorphisms} if
for every $g \in \rmS(\Omega)$ 
it holds
\begin{equation}\label{def:invbymap}
(X,Y) \in \Bb \quad \Longrightarrow\quad (g^*X,g^*Y)=(X\circ g,Y\circ g)  \in \Bb,
\end{equation}
where, with a slight abuse of notation, we denote with the same symbol $g^*$ both the pullback transformation induced by $g$ on $\cY$ and on $\cZ$.\\
A set $\Bb \subset \cY\times\cZ$ is \emph{law invariant} if it holds
\begin{equation}\label{def:invlaw}
(X,Y) \in \Bb,\quad (X',Y')\in\cY\times\cZ, \quad (X,Y)_\sharp\P=(X',Y')_\sharp\P\quad \Longrightarrow\quad
(X',Y') \in \Bb.
\end{equation}

A (single valued) operator $\map:\dom(\map)\subset \cY\to\cZ$ is \emph{invariant by measure-preserving
  isomorphisms}
(resp.~\emph{law invariant}) if its graph in $\cY\times \cZ$ is invariant by measure
preserving isomorphisms (resp.~law invariant).
\end{definition}
It is easy to check that a law invariant set or operator is also
invariant by measure-preserving isomorphisms.
It is also immediate to check that
an operator $\map:\dom(\map)\subset \cY\to\cZ$ is invariant by measure-preserving isomorphisms
if 
for every $g\in \rmS(\Omega)$
and every $X\in \dom(\map)$ it holds
\begin{equation}
  \label{eq:invariant-maps}
  X\circ g\in \dom(\map),\quad
  \map(X\circ g)=
  \map(X)\circ g.
\end{equation}
Similarly, $\map$ is law invariant if 
for every $X\in \dom(\map),\ X'\in \cY,\ Y'\in \cZ$
\begin{equation}
  \label{eq:invariant-maps2}
  (X,\map X)_\sharp \P=(X',Y')_\sharp \P\quad
\Longrightarrow\quad
X'\in \dom(\map),\quad
Y'=\map X'.
\end{equation}

\begin{remark}\label{rem:conj}
Notice that, when $\Z=\X^*$
are reflexive, $p>1$, and 
$q=p^*$ is the conjugate exponent of $p$, 
then $\cZ=\cY^*$ and 
the notion of invariance by measure-preserving isomorphisms for $\Bb\subset\cY\times\cZ$ coincides with the $\G$-invariance of
    Definition \ref{def:Ginv},
    $\G$ being the group of isometric isomorphisms 
    induced by $\rmS(\Omega)$ via 
    $g^*:\cY\times \cY^* \to \cY\times 
    \cY^*$ with $g^*(X,V) = (X\circ g,V\circ g)$ for every $(X,V) \in \cY\times \cY^*$ and every $g \in \rmS(\Omega)$.
\end{remark}

Let us denote by $\iota:\cY\to
\mathcal P_p(\Y)$ the push-forward operator, $\iota(X):=X_\sharp \P$ (cf.~\eqref{eq:pushf}). We frequently use the notation $\iota_X=\iota(X)$.
The map $\iota$ induces a one-to-one correspondence between
subsets in $\mathcal P_p(\Y)$ (see \eqref{condTanTX})  and law invariant subsets of $\cY$.

As a first result we show that for closed sets and continuous operators the two
notions of invariance are in fact equivalent.
\begin{proposition}[Closed sets invariant by m.p.i.~are law invariant]
  \label{prop:map-invariant}
  If $\Bb\subset \cY\times\cZ$ is invariant by
  measure-preserving isomorphisms, then its closure $\overline\Bb$ is also law invariant.
  In particular, a closed set of $\cY\times\cZ$  is law invariant if and only if it is invariant by measure-preserving isomorphisms and 
  if $\map:\dom(\map)\subset \cY\to\cZ$
  is a continuous operator invariant by measure-preserving isomorphisms and
  $\dom(\map)$ is closed, then
  $\map$ is also law invariant.
\end{proposition}
\begin{proof}
  Let us first observe that if $\Bb$ is invariant by measure-preserving isomorphisms then $\overline \Bb$ has the same property, since for every $g\in \rmS(\Omega)$ the pullback transformation $g^*$ is an isometry in $\cY$ and in $\cZ$.
  Let us now suppose that $\Bb$ is closed and invariant by measure-preserving isomorphisms,
  $(X,Y)\in \Bb$, $(X',Y')\in \cY\times\cZ$ with
  $(X,Y)_\sharp \P=(X',Y')_\sharp \P$.
  We can apply Corollary \ref{cor:from-gangbo} and find a sequence of
  measure-preserving isomorphisms $g_n\in \rmS(\Omega)$ such that $(X,Y)\circ g_n\to (X',Y')$ in $\cY\times\cZ$.
  Since $(X,Y)\circ g_n\in \Bb$
  and $\Bb$ is closed, we deduce that $(X',Y')\in \Bb$ as well.

  The case of continuous operators then follows by the fact that
  the graph of a continuous operator is closed.
\end{proof}

\begin{remark}\label{rem:equivinv} In light of Proposition \ref{prop:map-invariant}, if a subset $\Bb \subset \cY\times\cZ$ (resp.~an operator $\map:\dom(\map)\subset \cY\to\cZ$) is closed (resp.~continuous and $\dom(\map)$ is closed), we use the simplified terminology \emph{invariant}, whenever $\Bb$ (resp.~$\map$) is law invariant or invariant by measure-preserving isomorphisms, being these two notions equivalent.
\end{remark}

 As an application of the results in Section \ref{sec:appC},
 we obtain the following extension results.
 \begin{theorem}
 [Maximal extensions of monotone operators invariant by measure-preserving isomorphisms]
\label{thm:maximal-monotonicity}
Let $\Y$ be a separable and reflexive Banach space, $p\in (1,\infty)$, 
$\cY:=L^p(\Omega,\cB, \P;\Y)$. 
If $\Aa\subset \cY\times \cY^*$ is a 
monotone operator invariant by measure-preserving isomorphisms, 
 then there exists  
    a maximal monotone 
    extension of $\Aa$ which is 
    invariant by measure-preserving isomorphisms (and therefore also law invariant)
    whose domain is contained
    in $\clconv{\dom(\Aa)}.$
\end{theorem}
\begin{proof}
The thesis follows by applying Theorem \ref{thm:graziebauinv} and Remark \ref{rem:conj}. Recall also that if $\Aa$ is maximal monotone, then it is closed (see e.g.~\cite[Proposition 2.1]{Barbu10})
hence we can apply Proposition \ref{prop:map-invariant}.
\end{proof}
\begin{theorem}[Extension of Lipschitz and $\lambda$-dissipative invariant graphs]    
\label{cor:maximal-dissipativity}
Assume $p=2$ and 
let $\Y$ be a separable and Hilbert space.
The following hold:
\begin{enumerate}
    \item 
    if $\map:\dom(\map)\subset \cY\to\cY$ 
    is a $L$-Lipschitz function invariant by measure-preserving isomorphisms, then there exists a $L$-Lipschitz extension $ {\hat \map}:\cY\to\cY$, defined on the whole $\cY$ which is invariant by measure-preserving isomorphisms;
    \item if $\Bb \subset \cY \times \cY$ is a $\lambda$-dissipative operator
    which is invariant by measure-preserving isomorphisms, then there exists 
    a maximal $\lambda$-dissipative extension $\hat \Bb$ of $\Bb$ which is 
    invariant by measure-preserving isomorphisms (and therefore also law invariant)
    whose domain is contained in 
    $\clconv{\dom(\Bb)}.$
 \end{enumerate}
\end{theorem}
\begin{proof}
 The assertion is an immediate application
    of Theorems \ref{thm:inv-lip-ext}, \ref{thm:inv-max-ext},
    choosing $\H:=\cY$ and $\G_{\H}$ as the group of isometric isomorphisms 
    induced by $\rmS(\Omega)$ via $g^*:\cY \to \cY$ with $g^*X = X\circ g$ for every $X \in \cY$ and every $g \in \rmS(\Omega)$ (cf.~Remark \ref{rem:conj}). Notice that $\Bb$ is closed being maximal dissipative hence law invariant by Proposition \ref{prop:map-invariant}.
\end{proof}

We conclude this section
with some useful representation
results for various classes of law invariant transformations.

Given a set $D\subset \mathcal P_p(\Y)$, we set
\begin{equation}
  \label{eq:2}
  \Sp{\Y,D}:=\Big\{(x,\mu)\in \Y\times D\,:\,x\in \supp(\mu)\Big\},
\end{equation}
and we just write $\Sp{\Y}=\Sp{\Y,\mathcal P_p(\Y)}$.
The set $\Sp\Y$ is of $G_\delta$ type (see \cite[Formula (4.3)]{FSS22}) so that $\Sp{\Y,D}$ is a Borel set, if $D$ is Borel.

We state a first result on uniqueness of a representation of a single-valued operator $\map:\cY\to\cZ$ by a map from $\Sp{\Y}$ to $\Z$.
We then state existence (and uniqueness) of such map-representation when $\map$ is Lipschitz continuous and invariant, showing also further properties inherited by such map-representation. Finally, we will go back to the case of (possibly multivalued) maximal $\lambda$-dissipative operators in the final Theorem \ref{thm:invTOlawinv}.

\begin{lemma}\label{lem:uniqueness}
Let $\map: \dom(\map) \to \cZ$ be a map
defined in 
$\dom(\map) \subset \cY$.
If  $\ff_i: \Sp{\Y, \iota(\dom(\map)} \to \Z$, $i=1,2$, satisfy
\begin{itemize}
    \item 
    $\ff_i(\cdot, \mu)$ is continuous for every $\mu \in \iota(\dom(\map))$, 
    \item 
   for every $X\in \dom(\map)$
   $\map X(w)=\ff_i(X(w),X_\sharp\P)$ for a.e.~$\omega\in
    \Omega$,
\end{itemize} 
then $\ff_1=\ff_2$.
\end{lemma}
\begin{proof}  Let $(x,\mu) \in \Sp{\Y, \iota(\dom(\map)}$. Since $\mu \in \iota(\dom(\map))$, we can find (a representative of) $X \in \dom(\map)$ such that $X_\sharp \P= \mu$ and a full $\P$-measure set $\Omega_0 \subset \Omega$ such that
\[ \map X(\omega)=\ff_1(X(\omega),X_\sharp\P) = \ff_2(X(\omega),X_\sharp\P) \text{ for every } \omega\in \Omega_0.\]
Since $X(\Omega_0)$ is dense in $\supp(\mu)$ we can find $(\omega_n)_n \subset \Omega_0$ such that $X(\omega_n) \to x$. Using the continuity of $\ff_i(\cdot, \mu)$  we can write the above equality for $\omega=\omega_n$ and then pass to the limit as $n \to + \infty$ obtaining that $\ff_1(x,\mu)=\ff_2(x,\mu)$.
\end{proof}

In the following Theorem \ref{thm:continuous}, we 
 provide an important structural representation of 
 invariant Lipschitz maps from $\cY$ to $\cZ$. 
 A similar kind of problems have been considered in \cite[Proposition 5.36]{CD18}.
\begin{theorem}[Structure of invariant Lipschitz maps]
\label{thm:continuous}
Let $\Y,\Z$ be 
separable Banach spaces
and let $\cY,\cZ$ be as in 
\eqref{eq:XY-notation}.
\begin{enumerate}
\item Let $\map: \cY \to \cZ$ be a $L$-Lipschitz map, invariant by measure-preserving isomorphisms. Then there exists a unique continuous map $\euler\map:\Sp{\Y}
\to \Z$ such that 
\begin{equation}\label{eq:5bis}
 \text{for every $X\in \cY$,\quad
    $\map X(\omega)=\euler \map(X(\omega),X_\sharp\P)$ for a.e.~$\omega\in
    \Omega$.}
\end{equation}
Moreover $\euler \map (\cdot, \mu): \supp(\mu) \to \Z$ is $L$-Lipschitz.
\item Let 
$\Y$ be an Hilbert space and $p=2$, let $\dom(\map) \subset \cY$ and let $\map: \dom(\map) \to \cY$ be a $L$-Lipschitz map, invariant by measure-preserving isomorphisms. Then there exists a unique continuous map $\euler\map:\Sp{\Y, \iota(\dom(\map)}\to \Y$ such that 
\begin{equation}\label{eq:5}
 \text{for every $X\in \dom(\map)$,\quad
    $\map X(\omega)=\euler \map(X(\omega),X_\sharp\P)$ for a.e.~$\omega\in
    \Omega$.}
\end{equation}
Moreover $\euler \map (\cdot, \mu): \supp(\mu) \to \Y$ is $L$-Lipschitz.
\end{enumerate}
\end{theorem}
\begin{proof}
We prove item (1).
Let $\cN:= \{2^n \mid n \in \N\}$ and let $(\mathfrak P_N)_{N \in \cN}$ be a $\cN$-segmentation of $(\Omega, \cB, \P)$ as in Definition \ref{def:segm}, whose existence is granted by Lemma \ref{le:exseg}. Let us define 
  \[ \cB_n:= \sigma(\mathfrak P_{2^n}),\quad
    \cY_n:= L^p(\Omega, \cB_n, \P; \Y), \quad 
        \cZ_n:= L^q(\Omega, \cB_n, \P; \Z),\quad n \in \N.\]
  We divide the proof in several steps.

\smallskip \noindent
(a) \emph{If $X \in \cY_m$ for some $m \in \N$, then (there exists a
  unique representative of) $\map X$ (that) belongs to 
  $\cZ_m$ and
\begin{equation}\label{eq:perlafunz}
  |\map X(\omega')-\map X(\omega'')|_\Z\le L|X(\omega')-X(\omega'')|_\Y\quad
    \text{for every }\omega',\omega''\in \Omega.
  \end{equation}}%
Let $\Omega' \subset \Omega$ be a full $\P$-measure subset of $\Omega$
where both \eqref{eq:lebdiffgu} and Lemma \ref{le:separation} hold 
 for the $L^q(\Omega, \cB,
\P;\Z)$ function $\map X$.

Let us
fix $k \in I_m:=\{0,\dots,2^m-1\}$ and show that (a representative of)
$\map X $ is constant on $\Omega'_{m,k}:=
\Omega_{m,k} \cap \Omega'$, where $\mathfrak P_{2^m} :=
\{\Omega_{m,k}\}_{k \in I_m}$.

Let $\omega', \omega'' \in \Omega'_{m,k}$ with $\omega'\ne \omega''$. For every $n \in \N$ there exist $k(n;\omega'), k(n;\omega'') \in I_{n}$ such that $\omega' \in \Omega_{n,k(n;\omega')}$ and $\omega'' \in \Omega_{n,k(n;\omega'')}$. By Lemma \ref{le:separation} we know that for $n \in \N$ sufficiently large $\Omega_{n,k(n;\omega')}, \Omega_{n,k(n;\omega'')} \subset \Omega_{m,k}$ and $\Omega_{n,k(n;\omega')} \cap \Omega_{n,k(n;\omega'')} = \emptyset$. Thus, since $\P(\Omega_{n,k(n;\omega')})= \P(\Omega_{n,k(n;\omega'')})=2^{-n}$ for every $n \in \N$ (see Definition \ref{def:segm}), by Corollary \ref{cor:mpsub} we can find a measure-preserving isomorphism $g_n \in \rmS(\Omega)$ such that 
\[ (g_n)_\sharp \P |_{\Omega_{n,k(n;\omega')}} = \P |_{\Omega_{n,k(n;\omega'')}}\]
and $g_n$ is the identity outside $\Omega_{n,k(n;\omega')} \cup \Omega_{n,k(n;\omega'')}$. By \eqref{eq:invariant-maps} and by Lipschitz continuity of $\map$, we have
\[\left | \map X \circ g_n - \map X \right |_\cZ \le L\left | X \circ g_n - X \right |_\cY=0 \quad \text{for every
    integer $n$ sufficiently large}, \]
since $X$ is constant on the whole $\Omega_{m,k}$. This implies that
\[ 2^{-n} \int_{\Omega_{n,k(n;\omega')}} \map X \de \P = 2^{-n}
  \int_{\Omega_{n,k(n;\omega'')}} \map X \de \P  \quad \text{eventually}. \]
By definition of conditional expectation, this means that
\[ \E_\P \left [ \map X \mid \sigma \left (\mathfrak P_{2^n} \right )
  \right ] (\omega') = \E_\P \left [ \map X \mid \sigma \left
      (\mathfrak P_{2^n} \right ) \right ] (\omega'') \quad
  \text{eventually.}\]
Passing to the limit as $n \to + \infty$ we get by
\eqref{eq:lebdiffgu} that $\map X(\omega')=\map X(\omega'')$. This proves
that $\map X$ is $\P$-almost everywhere constant on $\Omega_{m,k}$;
being $k \in I_m$ arbitrary, we can find a representative of $\map X$
belonging to $\cZ_m$.
If $\omega', \omega'' \in \Omega$ and $\omega' \in \Omega_{m,i}$, $\omega''\in \Omega_{m,j}$, $i,j\in I_m$ we choose as $g \in \rmS(\Omega)$ a measure-preserving isomorphism induced by the permutation $\sigma \in \symg{I_m}$ that swaps $i$ and $j$ (see Corollary \ref{cor:isomor}), so that we get by Lipschitz continuity of $\map$ that
  \begin{displaymath}
    \frac 1{2^{(m-1)/2}} |\map X(\omega')-\map X(\omega'')|_\Z \le
    L \frac 1{2^{(m-1)/2   }} |X(\omega')-X(\omega'')|_\Y
  \end{displaymath}
  which yields \eqref{eq:perlafunz}.
  \smallskip \noindent
  
  (b) \emph{For every $X\in \cY$ there exists a unique $L$-Lipschitz
    map
    $\ff_X:\supp(X_\sharp \P)\to \Z$ such that
    $\map X(\omega) = \ff_{X}(X(\omega))$ for $\P$-a.e.~$\omega$.}

  Let $X \in \cY$;
  setting $X_n:=\E[X|\cB_n]\in \cY_n$, by Theorem \ref{theo:stroock}
  we have $X_n \to X$ (hence also $\map X_n \to \map X$). Let us consider two representatives of $\map X$ and $X$, a full
measure set $\Omega_0 \subset \Omega$ and a subsequence $(X_{n_k})_k$
s.t.~$X_{n_k}(\omega) \to X(\omega)$ and
$\map X_{n_k}(\omega) \to \map X (\omega)$ for every $\omega \in
\Omega_0$. By \eqref{eq:perlafunz} we have
\[ |\map X_n(\omega')-\map X_n(\omega) |_\Z \le
  L |X_n(\omega')-X_n(\omega)|_\Y
  \text{ for every }\omega, \omega' \in \Omega, \, n \in \N. \]
Passing to the limit in the above inequality for every couple $(\omega, \omega') \in \Omega_0^2$, we obtain that 
\[ |\map X(\omega')-\map X(\omega) |_\Z  \le L
  |X(\omega')-X(\omega)|_\Y \text{ for every }\omega, \omega' \in \Omega_0. \]
This gives the existence of a (unique) $L$-Lipschitz function
$\ff_X: \overline{X(\Omega_0)}\to \Z$
s.t.~$\map X(\omega) = \ff_{X}(X(\omega))$ for every $\omega \in
\Omega_0$. Notice that $\overline{X(\Omega_0)}\supset \supp(X_\sharp \P)$.
\smallskip \noindent

(c) \emph{If $X, X'\in \cY$ with $\mu=X_\sharp \P=X'_\sharp \P$, then $\ff_X=\ff_{X'}$ on $\supp(\mu)$. In particular, we can define
$\euler\map (\cdot,\mu):=\ff_X(\cdot)$ whenever $\mu=X_\sharp\P$.}

By hypothesis we have
\begin{equation}\label{eq:laprimacosa}
  (X, \map X)_\sharp \P = (X', \map X')_\sharp \P.
\end{equation}
By the previous claim,
the disintegration (see e.g. \cite[Theorem 5.3.1]{ags} for details on the disintegration theorem) of the common measure $\boldsymbol\mu:= (X, \map X)_\sharp
\P $
with respect to its first marginal $\mu$ is
given by $\delta_{\ff_X(\cdot)}$ which should coincide with
$\delta_{\ff_X'(\cdot)}$ $\mu$-a.e.~in $\Y$. Since $\ff_X$ and $\ff_{X'}$
are both
Lipschitz continuous and coincide $\mu$-a.e., they coincide on
$\supp(\mu)$.
\smallskip \noindent

(d)
\emph{The map $\euler\map$ is continuous on $\Sp \Y$.}

Let us consider a sequence $(x_n,\mu_n)_n$ in $\Sp\Y$ converging to
$(x,\mu)\in \Sp\Y$ and let us prove that there exists an increasing subsequence
$k\mapsto n(k)$ such that $\euler\map(x_{n(k)},\mu_{n(k)})\to
\euler\map(x,\mu)$ as $k\to\infty$.

By Proposition \ref{prop:pratelli}, we can find a limit map $X\in \cY$
and a sequence $X_n\in \cY$ converging to $X$ such that
$(X_n)_\sharp\P=\mu_n,\ X_\sharp \P=\mu$.
Since $\map X_n\to\map X$, 
we can then extract a subsequence $k\mapsto n(k)$
and find a set of full measure $\Omega_0\subset \Omega$ such that
$X_{n(k)}(\omega)\to X(\omega)$ and $Y_{n(k)}(\omega)\to
Y(\omega)$ as $k\to\infty$ for every $\omega\in
\Omega_0$, where $Y_n:=\map X_n,\ Y:=\map X$.

Let us fix $\eps>0$; since $x\in \supp(\mu)$ and $X(\Omega_0)\cap
\supp(\mu)$ is dense in $\supp(\mu)$, we can find
$\omega\in \Omega_0$ such that
$|X(\omega)-x|_\Y\le \eps$.
We then obtain
\begin{align*}
  |\euler\map(x_{n(k)},\mu_{n(k)})-\euler\map(x,\mu)|_\Z 
  &\le
    |\euler\map(x_{n(k)},\mu_{n(k)})-
    \euler\map(X_{n(k)}(\omega),\mu_{n(k)})|_\Z
  \\&\quad +
    |\euler\map(X_{n(k)}(\omega),\mu_{n(k)})-
  \euler\map(X(\omega),\mu)|_\Z
  \\&\quad+
  |\euler\map(X(\omega),\mu)-\euler\map(x,\mu)|_\Z
  \\&\le L|
  x_{n(k)}-X_{n(k)}(\omega)|_\Y+|Y_{n(k)}(\omega)-Y(\omega)|_\Z+
  L|X(\omega)-x|_\Y.
\end{align*}
Taking the $\limsup$ as $k\to\infty$ we get
\begin{displaymath}
  \limsup_{k\to\infty}
  |\euler\map(x_{n(k)},\mu_{n(k)})-\euler\map(x,\mu)|_\Z
  \le 2 L\eps
\end{displaymath}
and, since $\eps>0$ is arbitrary, we obtain the convergence.

\medskip \noindent
We prove item (2): by Theorem \ref{cor:maximal-dissipativity}, it is enough to prove the statement in case the map $\map$ is defined on the whole $\cY$. 
We can then apply the previous claim
with $\Z=\Y$ and $p=q=2.$
\end{proof}
In the particular case when $p=q=2$ and $\Y=\Z$ is a Hilbert space,
also $\lambda$-dissipativity \eqref{eq:140} is inherited from $\map$ to its representative map $\euler\map$. 

\begin{proposition}[$\lambda$-dissipative representations]
  \label{prop:monotonicity}
  Let us suppose that $p=2$, $\Y$ is a separable Hilbert space, and
  $\map:\cY\to \cY$ is an invariant Lipschitz $\lambda$-dissipative operator.
  Then, for every $\mu\in \mathcal P_2(\Y)$, the map $\euler \map(\cdot,\mu)$ of Theorem \ref{thm:continuous}
  is (pointwise) $\lambda$-dissipative, i.e.
  \begin{equation}
    \label{eq:6bis}
    \langle \euler \map(x,\mu)-\euler \map(x',\mu),x-x'\rangle_\Y\le \lambda |x-x'|^2
    \quad\text{for every }x,x'\in \supp(\mu).
  \end{equation}
\end{proposition}
\begin{remark}\label{rmk:lambdamaps}
Recalling Remark \ref{rem:transff}, $\map:\cY\to\cY$ is an invariant Lipschitz $\lambda$-dissipative operator if and only if its $\lambda$-transformation $\map^\lambda:=\map-\lambda \ii_\cY$ is an invariant Lipschitz dissipative operator.
Moreover, by applying Theorem \ref{thm:continuous} to both $\map$ and $\map^\lambda$, we can identify
\[\euler{\map^\lambda}(x,\mu)\equiv \euler \map(x,\mu)-\lambda x,\quad\text{for every }(x,\mu)\in\Sp{\Y},\]
so that both $\euler\map$ and $\euler{\map^\lambda}$ satisfy \eqref{eq:5} with $\map$ and $\map^\lambda$ respectively.
\end{remark}
\begin{proof}
Thanks to Remark \ref{rmk:lambdamaps}, proving the $\lambda$-dissipativity in \eqref{eq:6bis} for $\map$ is equivalent to prove the $0$-dissipativity result in \eqref{eq:6bis} for $\map^\lambda$.

  We keep the same notation of the proof of Theorem \ref{thm:continuous} applied to $\map^\lambda$.
  The case when $\mu=X_\sharp \P$, with $X\in \cY_m$ for some $m\in \N$,
  follows as in claim (a) of the proof of Theorem \ref{thm:continuous}: we can assume that
  $\map^\lambda X$ is constant on every $\Omega_{m,k}$. If
  $\omega',\omega''\in \Omega$ and $\omega'\in \Omega_{m,i}$,
  $\omega'' \in \Omega_{m,j}$, $i,j\in I_m$ and $g\in \rmS(\Omega)$
  is a measure-preserving isomorphism induced by the permutation that swaps
  $i$ and $j$ as in Corollary \ref{cor:isomor}, we get
  \begin{displaymath}
    \frac 1{2^{m-1}}\langle \map^\lambda X(\omega')-\map^\lambda
    X(\omega''),X(\omega')-X(\omega'')\rangle_\Y
    \le 0.
  \end{displaymath}
  We can eventually argue by approximation, as in claim (b) of the proof of Theorem \ref{thm:continuous}, and using
  the
  representation of $\map^\lambda X$ in terms of $\euler {\map^\lambda}$ to get
   \begin{equation}
    \label{eq:6}
    \langle \euler {\map^\lambda}(x,\mu)-\euler {\map^\lambda}(x',\mu),x-x'\rangle_\Y\le 0
    \quad\text{for every }x,x'\in \supp(\mu).
  \end{equation}
\end{proof}

\begin{proposition}[Stability of Lipschitz representations] If $\map_n, \map:\cY\to\cZ$ are $L$-Lipschitz invariant maps and $D\subset \cY$ is an invariant closed set,
  such that
  \begin{equation}
    \label{eq:15}
    \map_n X\to \map X\quad\text{for every }X\in D,
  \end{equation}
  then
  setting $\tilde D:=\big\{X_\sharp\P:X\in D\big\}$, we have
  $\euler{\map_n}\to \euler{\map}$
  pointwise in $\Sp{\Y,\tilde D}$, where $\euler \map$ is as in 
  Theorem \ref{thm:continuous}.
\end{proposition}
\begin{proof} We fix $\mu\in \tilde D$
and we set $\ff_n:=\euler{\map_n}(\cdot,\mu):\supp(\mu)\to \Z$.
We observe that $\ff_n$ are $L$-Lipschitz
and form a Cauchy sequence in $L^q(\Y,\mu;\Z)$: 
we denote by
$\tilde \ff$ its limit.
For every $x\in \supp(\mu)$ and $\rho>0$, we can set
\begin{equation}
  \label{eq:17}
  \ff_{n,\rho}(x):=\frac1{\mu(B(x,\rho))}\int_{B(x,\rho)}\ff_n(u)\,\d\mu(u),\quad
  \ff_{\rho}(x):=\frac1{\mu(B(x,\rho))}\int_{B(x,\rho)}\tilde \ff(u)\,\d\mu(u).
\end{equation}
We have $\ff_{n,\rho}(x)\to \ff_\rho(x)$ for every $x\in \supp(\mu)$ and
every $\rho>0$. On the other hand
\[|\ff_{n,\rho}(x)-\ff_n(x)|_\Z\le L\rho,\]
so that a simple argument using the triangle inequality shows that
the sequence $(\ff_n(x))_{n\in \N}$ is Cauchy in $\Z$ for every $x\in
\supp(\mu)$ and its pointwise limit $\gg(\cdot,\mu)$
is $L$-Lipschitz and represents
$\map$
as in \eqref{eq:5} for every $X \in D$. Arguing as in claim (d) in the proof of Theorem \ref{thm:continuous} and using the
continuity of $\map$, we can eventually deduce that $\gg$ is
continuous in $\Sp{\Y,\tilde D}$. The fact that $\gg=\euler \map$ comes from Lemma \ref{lem:uniqueness}.
\end{proof}
When the maps $\map_n$
are not uniformly Lipschitz, we can still obtain a limiting representation. 
\begin{lemma}
    \label{le:last-but-not-least}
     Let 
$D \subset \cY$ and $\map_n, \map: D \to \cZ$ be maps such that $\map_n$ are law invariant, $\map_n$ converge pointwise to $\map$ on $D$ 
as $n\to\infty$ and there exist Borel functions $\ff_n: \Sp{\Y, \iota(D)} \to \Z$ such that 
\[ \text{for every $n \in \N$ and every $X\in D$,\quad
    $\map_n X(\omega)=\ff_n(X(\omega),X_\sharp\P)$ for a.e.~$\omega\in
    \Omega$.}
\]
Then $\map$ is law invariant and, for every $\mu \in \iota(D)$, there exists a map $\ff[\mu] \in 
 L^q(\Y, \mu; \Z)$ such that, for every $X \in D$ there exists an increasing subsequence $k \mapsto n_k$ such that
\[ \lim_{k} \ff_{n_k}(X(\omega), X_\sharp \P) = \ff[X_\sharp \P](X(\omega)) = \map X (\omega) \text{ for a.e.~} \omega\in
    \Omega.\]
\end{lemma}
\begin{proof}
    First of all notice that $\map$ is law invariant, being pointwise limit of law invariant maps. Let $\mu \in \iota(D)$ be fixed. If $X\in D$ is such that $X_\sharp \P=\mu$, we can find an increasing subsequence $k \mapsto n_k$ such that $\map_{n_k} X \to \map X$ $\P$-a.e.~in $\Omega$ and define 
\[ \ff_X(x):= \begin{cases} \lim_{k \to + \infty}\ff_{n_k}(x, \mu) \quad &\text{ if } x \in E, \\ 0 \quad &\text{ else}, \end{cases}\]
where $E \subset \Y$ is the Borel set of points in $\Y$ where $\lim_{k \to + \infty}\ff_{n_k}(x, \mu)$ exists. Let us now show that $\ff_X$ represents $\map$, i.e.
\[ \text{
    $\map X(\omega)=\ff_X (X(\omega))$ for a.e.~$\omega\in
    \Omega$.}
\]
We know that 
\[ \text{for every $k \in \N$ \quad 
    $\map_{n_k} X(\omega)=\ff_{n_k}(X(\omega),X_\sharp\P)$ for a.e.~$\omega\in
    \Omega$,}
\]
so that, choosing representatives of $X$, $\map X$, $\map_{n_k}X$, we can find a full $\P$-measure set $\Omega_0 \subset \Omega$ such that
\[ \map X(\omega) = \lim_{k \to + \infty} \map_{n_k}X (\omega) = \lim_{k \to +\infty} \ff_{n_k}(X(\omega), X_\sharp \P) \quad \text{ for every }\omega \in \Omega_0.
\]
Thus, for every $\omega \in \Omega_0$, $X(\omega) \in E$ and then $\map X(\omega) = f_X(X(\omega))$. If now $X' \in D$ is such that ${X'}_\sharp \P=\mu$, arguing as in claim (c) of the proof of Theorem \ref{thm:continuous}, it is easy to see that $\ff_X=\ff_{X'}$ $\mu$-almost everywhere.
\end{proof}
\medskip

 We eventually apply the previous results to  (possibly multivalued) maximal $\lambda$-dissipative operators $\Bb\subset\cH\times\cH$ (with $p=q=2$ and $\Y=\Z$) which are invariant by measure-preserving isomorphisms. We show that the invariance property is inherited by the associated resolvent operator, Yosida approximation, minimal selection and semigroup which also enjoy a map-representation property.

We denote by $\dommmo:=\Big\{X_\sharp\P:X\in \dom(\Bb)\Big\}$ the image in $\mathcal P_2(\X)$
of the domain of
$\Bb$.
\begin{theorem}[Structure of resolvents, Yosida approximations and semigroups]
\label{thm:invTOlawinv}
  Let $\Y$ be a separable Hilbert space and let $p=2$. Let $\Bb \subset \cH \times \cH$ be a maximal $\lambda$-dissipative operator which is invariant by measure-preserving isomorphisms. Then for every $0<\tau<1/\lambda^+,\ t\ge0$
  the operators $ \Bb, \Bb_\tau, \resolvent\tau,\Sgp_t,\Bb^\circ$ are
  law invariant. Moreover there exist (uniquely defined) continuous maps $\jj_\tau:\Sp\X\to \X$, $\bb_\tau: \Sp \X \to \X$, and $\ss_t:\Sp{\X,\overline{\dommmo}}\to \X$ such that
  \begin{enumerate}
  \item \label{eq:lip1} for every $\mu \in \prob_2(\X)$, the map  $\jj_\tau(\cdot,\mu): \supp(\mu) \to \X$ is $(1-\lambda \tau)^{-1}$-Lipschitz continuous, for $0<\tau< 1/\lambda^+$,
   \item \label{eq:lip3} for every $\mu \in \prob_2(\X)$, the map  $\bb_\tau(\cdot,\mu): \supp(\mu) \to \X$ is $\frac{2-\lambda \tau}{\tau(1-\lambda \tau)}$-Lipschitz continuous, for $0<\tau< 1/\lambda^+$,
  \item \label{eq:lip2} for every $\mu \in \overline{\dommmo}$, the map $\ss_t(\cdot,\mu): \supp(\mu) \to \X$ is $e^{\lambda t}$-Lipschitz continuous,
  \end{enumerate}
  and
  \begin{align} \label{eq:7-1} \text{ for every $X \in \cH$, } \resolvent\tau X(\omega)=\jj_\tau(X(\omega),X_\sharp\P) \text{ for $\P$-a.e.~$\omega \in \Omega$,}\\
   \label{eq:7-3} \text{ for every $X \in \cH$, } \Bb_\tau X(\omega)=\bb_\tau(X(\omega),X_\sharp\P) \text{ for $\P$-a.e.~$\omega \in \Omega$,}\\
  \label{eq:7-2} \text{ for every $X \in \overline{\dom(\Bb)}$, } \Sgp_t X(\omega)=\ss_t(X(\omega),X_\sharp\P) \text{ for $\P$-a.e.~$\omega \in \Omega$,}
  \end{align}
  together with the invariance and semigroup properties
  \begin{equation}
    \label{eq:8}
    \begin{gathered}
      \mu\in \overline{\dommmo}\quad\Rightarrow\quad
      \ss_t(\cdot,\mu)_\sharp\mu\in  \overline{\dommmo};\qquad
      \mu\in {\dommmo}\quad\Rightarrow\quad
      \ss_t(\cdot,\mu)_\sharp\mu\in  {\dommmo},\\
      \ss_{t+h}(x,\mu)=\ss_h(\ss_t(x,\mu),\ss_t(\cdot,\mu)_\sharp\mu)\quad
      \text{for every }(x,\mu) \in \Sp{\X,\overline{\dommmo}},\quad t,h\ge0.
    \end{gathered}
  \end{equation}
  Finally, for every $\mu\in \dommmo$, there exists a map
  $\bb^\circ(\cdot,\mu) \in L^2(\X, \mu; \X)$ such that
  for every $X\in \cH$
  \begin{equation}
    \label{eq:9}
    \text{ if $X_\sharp \P=\mu$ then $X \in \dom(\Bb)$, } \Bb^\circ X(\omega)=\bb^\circ(X(\omega),\mu) \text{ for $\P$-a.e.~$\omega \in \Omega$.}
  \end{equation}
  For every $\mu \in \dommmo$,
  the map $\bb^\circ(\cdot,\mu)$ is $\lambda$-dissipative in a set $\X_0\subset \X$
  of full $\mu$-measure and satisfies
  \begin{equation}
    \label{eq:10}
    \lim_{h\downarrow0}\int\bigg|\frac 1h(\ss_{t+h}(x,\mu)-\ss_t(x,\mu))-
    \bb^\circ(\ss_t(x,\mu),\ss_t(\cdot,\mu)_\sharp \mu)\bigg|^2\,\d\mu(x)=0\quad
    t\ge0.
  \end{equation}
\end{theorem}
\begin{remark}
  \label{rem:aggiungi}
By Theorem \ref{thm:invTOlawinv}, a maximal $\lambda$-dissipative operator $\Bb \subset \cH \times \cH$, $\lambda\in\R$, is law invariant if and only if it is invariant by measure-preserving isomorphisms.
Thanks to \eqref{eq:9}, also $D(\Bb)$ is \emph{law invariant}, i.e. if $X\in D(\Bb)$ and $Y\in\cH$ is such that $Y_\sharp \P=X_\sharp\P$, then $Y\in D(\Bb)$.
\end{remark}

\begin{proof}
 First, notice that $\Bb$ is closed being maximal $\lambda$-dissipative, hence it is law invariant by Proposition \ref{prop:map-invariant}.
Recall that $\resolvent\tau$ is everywhere defined, $(1-\lambda \tau)^{-1}$-Lipschitz continuous for every $0<\tau<1/\lambda^+$ (see Section \ref{subsec:extension-diss}) and invariant by measure-preserving isomorphisms by Proposition \ref{prop:resdiscr}. Fixed $0<\tau<1/\lambda^+$, we can thus apply Proposition \ref{prop:map-invariant} and Theorem \ref{thm:continuous}(1) and get that $\resolvent\tau$ is law invariant together with property \eqref{eq:lip1} in Theorem \ref{thm:invTOlawinv} and \eqref{eq:7-1}. Similarly, $\Sgp_t$ is defined on the closed set $\overline{\dom(\Bb)}$, it is $e^{\lambda t}$-Lipschitz continuous (cf.~Section \ref{subsec:extension-diss})  and invariant by measure-preserving isomorphisms by Proposition \ref{prop:resdiscr}, so that we can apply Proposition \ref{prop:map-invariant} and Theorem \ref{thm:continuous}(1) and get that it is law invariant together with property \eqref{eq:lip2} of Theorem \ref{thm:invTOlawinv} and \eqref{eq:7-2}. The content of \eqref{eq:8} immediately follows by the semigroup and invariance properties of $\Sgp_t$ (cf.~Section \ref{subsec:extension-diss}), also using that $\overline{\iota(\dom(\Bb))}=\iota(\overline{\dom(\Bb)})$.

We now prove \eqref{eq:9}. If $X \in \dom(\Bb)$, we have that $\Bb_\tau X \to \Bb^\circ X$ as $\tau \downarrow 0$, moreover $\Bb_\tau$ is law invariant, everywhere defined and, for every $0<\tau< 1/\lambda^+$, $\Bb_\tau$ is $\lambda/(1-\lambda\tau)$-dissipative and $\frac{2-\lambda\tau}{\tau(1-\lambda\tau)}$-Lipschitz continuous (cf.~Section \ref{subsec:extension-diss}). Hence, can apply Theorem \ref{thm:continuous}(2) and get that $\Bb^\circ$ is law invariant and that there exists, for every $\mu \in \dommmo$, a map $\bb^\circ[\mu]\equiv \bb^\circ(\cdot,\mu)\in L^2(\X,\mu; \X)$ such that for every $X \in \dom(\Bb)$
\begin{align}
 \Bb^\circ X (\omega) = &\bb^\circ[\mu](X(\omega)) \text{ for $\P$-a.e.~$\omega \in \Omega$},\\ \label{eq:conv2}
\text{there exists $\tau_k \downarrow 0$ s.t. } &\bb_{\tau_k}(X(\omega), \mu) \to \bb^{\circ}[\mu](X(\omega)) \text{ for $\P$-a.e.~$\omega \in \Omega$},
\end{align}
where $\bb_{\tau}$ is the (unique) continuous map that represents $\Bb_{\tau}$ coming from Theorem \ref{thm:continuous}(1), $0<\tau< 1/\lambda^+$. 
To complete the proof of \eqref{eq:9}, we have to check that, if $\mu\in \dommmo$ and $X\in\cH$ is such that $X_\sharp\P=\mu$, then $X\in\dom(\Bb)$. Since $\mu\in \dommmo$, there exists $Y\in\dom(\Bb)$ such that $Y_\sharp\P=X_\sharp\P=\mu$. By \eqref{eq:7-3}, we have
\begin{equation*}
|\Bb_\tau Y|^2_{\cH}=\int |b_\tau(Y(\omega),\mu)|^2\d\P(\omega)=\int |b_\tau(x,\mu)|^2\d\mu(x)=\int |b_\tau(X(\omega),\mu)|^2\d\P(\omega)=|\Bb_\tau X|_{\cH}^2.
\end{equation*}
Hence, since $Y\in\dom(\Bb)$, by \eqref{eq:btaucircD} we have
\[(1-\lambda\tau)|\Bb_\tau X|_{\cH}=(1-\lambda\tau)|\Bb_\tau Y|_{\cH}\uparrow |\Bb^\circ Y|_{\cH}<+\infty,\quad\text{as }\tau\downarrow 0.\]
Recalling \eqref{eq:btaucircN}, we get $X\in\dom(\Bb)$.
Finally \eqref{eq:10} follows by \eqref{eq:paramev} using \eqref{eq:7-2} and \eqref{eq:9}.

It only remains to show that, for every $\mu \in \dommmo$, the map $\bb^\circ[\mu]$ is $\lambda$-dissipative in a full $\mu$-measure set. To this aim, observe that we can apply Theorem \ref{thm:continuous}(1) to $\Bb_\tau$, $0<\tau<1/\lambda^+$, so that for every $\mu \in \prob_2(\X)$ we have
\begin{equation} \label{eq:dissy}
 \la \bb_\tau(y)-\bb_\tau(y'), y-y' \ra \le \frac{\lambda}{1-\lambda\tau} |y-y'|^2 \quad \text{ for every } y,y' \in \supp(\mu).
\end{equation}
Let $\mu \in \dommmo$ and let us consider a representative of $X \in \dom(\Bb)$ such that $X_\sharp \P=\mu$; let $\tau_k$ be a sequence as in \eqref{eq:conv2} and let $\Omega_0$ be a full $\P$-measure set where the convergence in \eqref{eq:conv2} takes place. If we take $y,y' \in X(\Omega_0) \cap \supp(\mu)$ then we can find $\omega, \omega' \in \Omega_0$ such that $X(\omega)=y$ and $X(\omega')=y'$ so that, passing to the limit as $k \to + \infty$ in \eqref{eq:dissy} written for $\tau=\tau_k$ we get that
\[\la \bb^{\circ}[\mu](y)-\bb^\circ[\mu](y'), y-y' \ra \le \lambda |y-y'|^2 \quad \text{ for every } y,y' \in X(\Omega_0) \cap \supp(\mu). \]
It is then enough to observe that $X(\Omega_0) \cap \supp(\mu)$ contains a Borel set $\X_0$ of full $\mu$-measure to conclude: in fact, being $X(\Omega_0)$ a Suslin set, we can find two Borel sets $E_0, E_1$ such that $E_0 \subset X(\Omega_0) \subset E_1$ and $\mu(E_1 \setminus E_0)=0$. Since $\mu(E_1)=X_\sharp \P (E_1) \ge \P(\Omega_0)=1$, we conclude that we can take $\X_0:=E_0 \cap \supp(\mu)$.
\end{proof}

\appendix
\section{An alternative proof of
the extension theorem for
\texorpdfstring{$\G$}{G}-invariant Lipschitz maps}
 We provide an alternative proof to the extension result for invariant Lipschitz maps stated in Theorem \ref{thm:inv-lip-ext}. Here we use a recent and beautiful explicit construction of 
a Lipschitz extension provided by 
\cite{ALM21} (see also \cite{ALM18}) and we show that it preserves the invariance.
We consider a Hilbert space $\H$ with norm $|\cdot|$ and scalar product $\la \cdot, \cdot\ra$ as in Section \ref{subsec:invariant-Lip}.

Recall that if $g:\W\to \R$
is a function defined 
in a Hilbert space $\W$, then
its convex envelope is defined by
\begin{equation}
    \label{eq:conv-env}
    \conv{g}(w):=
    \inf\Big\{\sum_{i=1}^N \alpha_i g(w_i):
    \alpha_i\ge0,\
    \sum_{i=1}^N\alpha_i=1,\ 
    w_i\in \W,\ 
    \sum_{i=1}^N \alpha_i w_i=w,\ N\in \N\Big\}.    
\end{equation}
If $g$ is locally bounded from above, then also $\conv g$ is locally bounded from above and it is therefore locally Lipschitz.
\begin{theorem}[\cite{ALM21}]\label{thm:ALM21}
    Let $f:D\to \H$ be 
    a $L$-Lipschitz map defined in $D\subset \H$.
    Setting for every $x,y\in \H$
    \begin{equation}
        \label{eq:first-step}
        \begin{aligned}
        g(x,y):={}&
        \inf_{x'\in D}
        \Big\{\langle f(x'),y\rangle+
        \frac L2 |(x-x',y)|_{\H\times \H}^2\Big\}+
        \frac L2 |(x,y)|_{\H\times \H}^2,\\
        \tilde g:={}&\conv g,
        \end{aligned}
    \end{equation}
    then $\tilde g$ is a convex function of class $\rmC^{1,1}$ in $\H\times \H$
    and its partial differential with respect to the second variable in $\H$
    \begin{equation}
        \label{eq:second-step}
        F(x):=\nabla_y \tilde g(x,0), \quad x \in \H,
    \end{equation}
    is a $L$-Lipschitz extension of $f$.
\end{theorem}
We state our first result concerning the extension of $\G_\H$-invariant
Lipschitz maps.
\begin{theorem}
    \label{thm:inv-lip-ext-app}
    Under the same assumption of Theorem \ref{thm:ALM21}, let us also suppose that $f$ is $\G_\H$-invariant
    according to \eqref{eq:G-invariance-map}.
    Then $F$ is $\G_\H$-invariant as well.
    In particular, any $\G_\H$-invariant $L$-Lipschitz function $f:D\to\H$ 
    defined in a subset $D$ of $\H$ admits
    a $\G_\H$-invariant $L$-Lipschitz extension $F:\H\to\H.$
\end{theorem}
\begin{proof} 
    We divide the proof in several claims.

    \noindent\smallskip
    \emph{Claim 1: the map $g$ is $\G_\H$-invariant, i.e.~$g(Ux,Uy)=g(x,y)$
    for every $(x,y)\in \H\times \H$
    and $U\in \G_\H.$}

    Every element $x'\in D$
    can be written as $x'=Ux''$
    with $x''=U^{-1}x'\in D$, so that
    for every $U\in \G_\H$ and $(x,y)\in \H\times \H$
    \begin{align*}
        g(Ux,Uy)&=
        \inf_{x'\in D}
        \Big\{\langle f(x'),Uy\rangle+
        \frac L2 |(Ux-x',Uy)|_{\H\times \H}^2\Big\}+
        \frac L2 |(Ux,Uy)|_{\H\times \H}^2
        \\&=
        \inf_{x''\in D}
        \Big\{\langle f(Ux''),Uy\rangle+
        \frac L2 |(U(x-x''),Uy)|_{\H\times \H}^2\Big\}+
        \frac L2 |(Ux,Uy)|_{\H\times \H}^2
        \\&=
        \inf_{x''\in D}
        \Big\{\langle  f(x''),y\rangle+
        \frac L2 |(x-x'',y)|_{\H\times \H}^2\Big\}+
        \frac L2 |(x,y)|_{\H\times \H}^2
        \\&=g(x,y)
    \end{align*}
    where we used 
    \eqref{eq:G-invariance-map}
    and the isometric character of $U$ 
    to get
    $\langle f(Ux''),Uy\rangle=
    \langle Uf(x''),Uy\rangle=
    \langle f(x''),y\rangle$.

    \noindent\smallskip
    \emph{Claim 2: the map $\tilde g:=\conv g$
    is $\G_\H$-invariant as well.}

    It is sufficient to observe that 
    for every $U\in \G_\H,$ $N\in \N$,
    and $\alpha_i\ge0$ with $\sum_{i=1}^N\alpha_i=1$,
    a collection 
    $\{(x_i,y_i)\}_{i=1}^N
    \in (\H\times \H)^N$
    satisfies $\sum_{i=1}^N \alpha_i 
    (x_i,y_i)=(x,y)$
    if and only if 
    $\sum_{i=1}^N \alpha_i (U x_i,Uy_i)=(Ux,Uy)$.
    Using \eqref{eq:conv-env}
    and the invariance of $g$ we 
    thus obtain $\tilde g(Ux,Uy)=\tilde g(x,y)$ for every $x,y \in \H$.

    \noindent\smallskip
    \emph{Claim 3: the map $F:=\nabla_y \tilde{g}(\cdot,0)$
    is $\G_\H$-invariant.}
    
    Since we know that $\tilde g$ is
    Frech\'et differentiable, 
    we observe that 
    $z= F(x)$
    if and only if 
    \begin{displaymath}
        \tilde g(x,y)=
        \tilde g(x,0)+\langle z,y\rangle
        +o(|y|)\quad
        \text{as }y\to 0.
    \end{displaymath}
    Being $\tilde g$ invariant,
        for every $U\in \G_\H$ and 
        $x\in \H$ the above formula immediately yields
    \begin{align*}    
         \tilde g(Ux,y)&=
         \tilde g(Ux,UU^{-1}y)=
         \tilde g(x,U^{-1}y)
         \\&=
        \tilde g(x,0)+\langle z,U^{-1}y\rangle
        +o(|U^{-1}y|)
        \\&=
        \tilde g(Ux,0)+\langle Uz,y\rangle
        +o(|y|)
        \quad
        \text{as }y\to 0,
    \end{align*}
    so that $F(Ux)=Uz=UF(x)$.
\end{proof}

\printbibliography

\end{document}